\documentclass{amsart}
\usepackage{amssymb,amsfonts,dsfont}
\usepackage{graphicx}
\usepackage{amsmath}
\usepackage{latexsym}
\usepackage{color}
\usepackage{comment}

\newtheorem{theorem}{Theorem}[section]
\theoremstyle{plain}
\newtheorem{lemma}[theorem]{Lemma}
\newtheorem{proposition}[theorem]{Proposition}

\newtheorem{corollary}[theorem]{Corollary}

\theoremstyle{remark}
\newtheorem{definition}[theorem]{Definition}
\newtheorem{example}[theorem]{Example}

\newtheorem{remark}[theorem]{Remark}

\numberwithin{equation}{section}

\DeclareMathOperator{\supp}{supp}

\newcommand{\M}{\mathcal{M}}

\newcommand{\Complex}{\mathbb{C}}

\newcommand{\abs}[1]{\left\vert#1\right\vert}

\newcommand{\norme}[1]{\|#1\|_E}
\newcommand{\normcomm}[1]{\|#1\|_{E(\textit{M},\tau)}}
\newcommand{\Ker}[1]{\text{Ker}\,#1}
\newcommand{\Kerp}[1]{\text{Ker}^\perp#1}
\newcommand{\nonsp}{E(\mathcal{M},\tau)}
\newcommand{\Rep}[1]{\text{Re}\,#1}

\newcommand{\Imp}[1]{\text{Im}\,#1}

\def\underset#1\to#2{\mathop{#2}\limits_{#1}{ }}
\def\overset#1\to#2{\mathop{#2}\limits^{#1}{ }}

\newcommand{\one}{\textup{\textbf{1}}}
\newcommand{\tauone}{\tau(\textup{\textbf{1}})}
\newcommand{\Mtau}{\left(\mathcal{M},\tau\right)}

\begin{document}
\title[$k$-extreme points in symmetric spaces of measurable operators]
{$k$-extreme points in symmetric spaces of measurable operators}

\author{M.M. Czerwi\'nska}
\address{ Department of Mathematics and Statistics, University of North Florida, Jacksonville, FL 32224} \email{m.czerwinska@unf.edu}

\author
{A. Kami\'nska}
\address{Department of Mathematical Sciences, The University of
Memphis, Memphis, TN 38152} \email{kaminska@memphis.edu}

\thanks {\emph{2010 subject classification}\ {46L52, 46L51, 46E30, 46B20, 46B28, 47L05,
47L20}}


\begin{abstract}
Let $\M$ be a semifinite von Neumann algebra with a faithful, normal, semifinite trace $\tau$ and $E$ be a strongly symmetric Banach function space on $[0,\tauone)$. We show that an operator $x$ in the unit sphere of $\nonsp$ is $k$-extreme, $k\in\mathbb N$, whenever its singular value function $\mu(x)$ is $k$-extreme and one of the following conditions hold (i) $\mu(\infty,x)=\lim_{t\to\infty}\mu(t,x)=0$ or (ii) $n(x)\M n(x^*)=0$ and $\abs{x}\geq \mu(\infty,x)s(x)$, where $n(x)$ and $s(x)$ are null and support projections of $x$, respectively. The converse is true whenever $\M$ is non-atomic. The global $k$-rotundity property follows, that is if $\M$ is non-atomic then $E$ is $k$-rotund if and only if $\nonsp$ is $k$-rotund.
As a consequence of the noncommutive results we obtain that $f$ is a $k$-extreme point of the unit ball of the strongly symmetric function space $E$ if and only if its decreasing rearrangement $\mu(f)$  is $k$-extreme and $\abs{f}\geq \mu(\infty,f)$. We conclude with the corollary on orbits $\Omega(g)$ and $\Omega'(g)$. We get that $f$ is a $k$-extreme point of the orbit $\Omega(g)$, $g\in L_1+L_{\infty}$, or $\Omega'(g)$, $g\in L_1[0,\alpha)$, $\alpha<\infty$, if and only if $\mu(f)=\mu(g)$ and $\abs{f}\geq \mu(\infty,f)$. From this we obtain a characterization of $k$-extreme points in Marcinkiewicz spaces.
\end{abstract}

\maketitle

\renewcommand{\abstractname}{Acknowledgement}
\begin{abstract}
The final publication is available at Springer via\\ http://dx.doi.org/10.1007/s00020-014-2206-1
\end{abstract}
\section{Preliminaries}

The purpose of the article is to characterize $k$-extreme elements of the symmetric spaces of measurable operators  in terms of their singular value functions. There have been many results  in the noncommutative spaces $\nonsp$, in particular in the noncommutative $L_p$ spaces, aiming to reduce the study on their geometric properties to the commutative settings. The first research in this direction can be attributed to Arazy \cite{A}, who described extreme points in the unitary matrix spaces. In the symmetric spaces of measurable operators extreme points were characterized by Chillin, Krygin and Sukochev in \cite{CKSextreme}.

The theory of noncommutative spaces $\nonsp$ has  attracted many mathematicians who continue to research their geometric properties. There is a long list of papers relating geometric properties of the operators and their  decreasing rearrangements, see for example results on local uniform rotundity and uniform rotundity \cite{CKSlur}, complex extreme points \cite{moj}, Kadets-Klee property   \cite{CDS}, or smooth points in \cite{CKK1}.  Worth mentioning is also work in \cite{R} investigating  local geometrical aspects of subspaces of noncommutative $L_p$ spaces associated with general von Neumann algebras.

Let $\M$ be a semifinite von Neumann algebra on the Hilbert space $H$ with given faithful normal semifinite trace $\tau$. The symbol $\one$ will be used to denote the identity in $\M$ and $\mathcal{P}(\M)$ will stand for the set of all projections in $\M$. It is known that  $\mathcal{P}(\M)$ is a complete lattice, that is the supremum and infimum exist for any non-empty subset of $\mathcal{P}(\M)$. Given $p,q\in \mathcal{P}(\M)$, we will write $p\vee q$ and $p\wedge q$ to denote the supremum and infimum of $p$ and $q$, respectively.
The projections $p$ and $q$ are said to be equivalent (relative to the von Neumann algebra $\M$) denoted by $p\sim q$, if there exists a partial isometry $v\in\ M$ such that $p =v^*v$ and $q = vv^*$. A non-zero projection $p\in
\mathcal{P}(\mathcal{M})$ is called \emph{minimal} if $q\in \mathcal{P}(\mathcal{M})$ and $q\leq p$ imply that $q=p$ or
$q=0$. The von Neumann algebra $\M$ is called \emph{non-atomic} if it has no minimal  projections. A projection $p \in P(\mathcal{M})$ is called \emph{$\sigma$-finite} (with respect to the trace $\tau$) if there exists a sequence $\{p_n\}$ in $P(\mathcal{M})$ such that $p_n\uparrow p$ and $\tau(p_n)<\infty$. If the unit element $\one$ in $\mathcal{M}$ is $\sigma$-finite, then we say that the trace $\tau$ on $\mathcal{M}$ is $\sigma$-finite \cite{KR, Takesaki, noncomm}. 

Given a self-adjoint (possibly unbounded) linear operator $x:\mathcal{D}(x)\subset H\to H$, we denote by $e^x(\cdot)$ the spectral measure of $x$. 
A closed, densely defined operator $x$, which commutes with all the unitary operators from the commutant $\M'$ of $\M$, is said to be \textit{affiliated} with $\M$. If in addition there exists $\lambda>0$ such that $\tau\left(e^{\abs{x}}(\lambda,\infty)\right)<\infty$ then $x$ is called $\tau$-\textit{measurable}. The collection of all $\tau$-measurable operators is denoted by $S\Mtau$.

 For an operator $x\in S\Mtau  $, the function $\mu(x)=\mu(\cdot,x):[0,\infty)\to[0,\infty]$ defined by
\[
\mu(t,x)=\inf\left\{s\geq 0: \tau(e^{|x|}(s,\infty))\leq t\right\},  \quad t>0,
\]
 is called a \emph{decreasing rearrangement} of $x$ or a \emph{generalized singular value function} of $x$. It follows that $\mu(x)$ is decreasing, right-continuous function on the real half-line. Note that in this paper the terms decreasing or increasing will always mean non-increasing or non-decreasing, respectively.
 Observe that if $\tauone<\infty$ then $\mu(t,x)=0$ for all $t\geq \tauone$, and so $\mu(\infty, x)=\lim_{t\rightarrow \infty}\mu(t,x)=0$.
We denote by $S_0\Mtau$ the collection of all $x\in S\Mtau$ for which  $\displaystyle\mu(\infty,x)=0$. $S_0\Mtau$
is a $*$-subalgebra in $S\Mtau $.

By a \emph{positive operator} $x$ we mean a self-adjoint operator such that $\left\langle x\xi,\xi\right\rangle \geqslant 0$ for all $\xi$ in the domain of $x$.
$S^+\Mtau$ will stand for the cone of positive measurable operators.

Any closed, densely defined linear operator $x$  has the \emph{polar decomposition} $x=u\abs{x}$, where $u$ is a partial isometry with kernel $\Ker u=\Ker x$. Moreover, the polar decomposition of $x^*$ is given by $x^*=u^*\abs{x^*}$. We have that $u^*u=s(x)=e^{\abs{x}}(0,\infty)$ and $uu^*=s(x^*)=e^{\abs{x^*}}(0,\infty)$, where \emph{support projections} $s(x)$ and $s(x^*)$ are projections onto $\Kerp{x}$ and $\Kerp{x^*}$, respectively. It is known that  $x$ is $\tau$-measurable if and only if $u\in\mathcal{M}$ and $\abs{x}$ is $\tau$-measurable. The \emph{null projection} $n(x)=\textbf{1}-s(x)=e^{\abs{x}}\{0\}$ is a projection onto $\Ker{x}$.

Let $L^0=L^0[0,\alpha)$ stand for the space of all complex-valued Lebesgue measurable functions
on $[0, \alpha)$, $\alpha\leq\infty$, (with identification a.e. with respect to the Lebesgue measure $m$). 

Given $f\in L^0$, the \emph{distribution function} $d(f)$ of $f$ is defined by $d(\lambda,f)=m\{t > 0:\,\abs{f (t)} > \lambda\}$, for all $\lambda\geq0$. The \emph{decreasing rearrangement} of $f$ is given by $\mu(t,f) = \inf\{s > 0:\,d(s,f)\leq t\}$,  $t \geq 0$. Observe that $d(f)=d(\cdot,f)$ and $\mu(f)=\mu(\cdot,f)$ are right-continuous, decreasing functions on $[0,\infty)$.
A \emph{support} of function $f\in E$, that is the set $\{t\in[0,\alpha):\,f(t)\neq 0\}$ will be denoted by $\supp f$.  Moreover for $f,g\in L^0$, we will write $f\prec g$ if $\int_0^t \mu(f)\leq \int_0^t\mu(g)$ for all $t>0$. For operators $x,y\in S\Mtau$, $x\prec y$  denotes $\mu(x)\prec \mu(y)$.

 A Banach space $E\subset L^0$ is called \emph{symmetric} (also called \emph{rearrangement invariant}) if it follows from $f\in L^0$, $g\in E$ and $\mu(f)\leq \mu(g)$ that $f\in E$ and $\norme{f}\leq\norme{g}$ \cite{BS, KPS}. It is well known that for every symmetric space $E$ on $[0,\alpha)$, we have that $E\subset L_1[0,\alpha)+L_{\infty}[0,\alpha)$. Moreover, if from $f,g\in E$ and $f\prec g$ we have that $\norme{f}\leq\norme{g}$ then $E$ is called a \emph{strongly symmetric} function space. Any symmetric space which is order continuous or satisfies the Fatou property is strongly symmetric \cite{BS,KPS}. Throughout the remainder of the paper we will assume that $E$ is strongly symmetric function space.

Given a semifinite von Neumann algebra $\Mtau$ and a symmetric Banach function space $E$ on $[0,\alpha)$, $\alpha=\tauone$, the corresponding \emph{noncommutative space} $E\Mtau$ is defined by setting
\[
E\Mtau=\{x\in S\Mtau:\quad \mu(x)\in E\}. 
\]
Observe that for any $x\in S\Mtau$, $\mu(x)=\mu(x)\chi_{[0,\alpha)}$, so we identify those two functions.
  $E\Mtau$ equipped with the norm $\normcomm{x}=\norme{\mu(x)}$ is a Banach space \cite{KS} and it is called \emph{symmetric space of measurable operators} associated with $\Mtau$ and corresponding to $E$. 

The trace $\tau$ on $\mathcal{M}^{+}$ extends uniquely to an additive, positively homogeneous, unitarily invariant and normal functional $\tilde{\tau}: S\Mtau^{+}\to[0,\infty]$, which is given by $\tilde{\tau}(x)=\int_0^{\infty}\mu(x)$, $x\in S\Mtau^+$. This extension is also denoted by $\tau$. The simple observation that an operator $x\in\nonsp$ is trace commuting with any projections $p\in P(\M)$, that is $\tau(xp)=\tau(px)$, will be used further in the paper \cite{DDP4}. 

If $E=L_p$ on $[0,\tauone)$, $1\leqslant p<\infty$, then for $x\in L_p\Mtau$ we have $\|x\|_{L_p\Mtau}=\|\mu(x)\|_{L_p}=\left(\int_0^{\tauone} \mu(\abs{x}^p) \right)^{1/p}=\left(\tau(\abs{x}^p)\right)^{1/p}$. The space $L_1\Mtau+\M$ is the space of all operators $x\in S\Mtau$ for which $\int_0^1\mu(x)<\infty$. By \cite[Proposition 2.6]{DDP4} we have that $E\Mtau\subset L_1\Mtau+\M$.

Consider a commutative von Neumann algebra $\mathcal{N}=\{N_f: L_2[0,\tauone)\to
L_2[0,\tauone):\quad f\in L_\infty[0,\tauone)\}$, where $N_f$ acts via pointwise multiplication on $L_2[0,\tauone)$ and the trace $\eta$ is given by integration, that is  $N_f(g)=f\cdot g$, $g\in L_2[0,\tauone)$, and  $\eta(N_f)=\int_0^{\tauone}fdm$, where $m$ is the Lebesgue measure on $[0,\tauone)$. 
This von Neumann algebra is commonly identified with $L_{\infty}[0,\tauone)$. $S(\mathcal{N},\eta)$ is identified with the set of all measurable functions on $R^+$ which are bounded except on a set of finite measure, denoted in this paper by $S([0,\tauone),m)$. Moreover for $N_f\in S(\mathcal{N},\eta)$, the generalized singular value function $\mu( N_f )$ is precisely the decreasing rearrangement $\mu(f)$ of the function $f\in S([0,\tauone),m)$. For any symmetric function space $E$ we have that $E(\mathcal{N},\eta)$ is isometrically isomorphic to the function space $E$.

For the theory of operator algebras we refer to
\cite{KR,Takesaki}, and for noncommutative Banach function spaces
to \cite{DDPnoncomm,noncomm,P}.


For the readers' convenience we will recall below the well known result on commuting elements of $S\Mtau$.
\begin{proposition}\cite{noncomm}
\label{prop:commute}
Two self-adjoint elements $a,b\in S\Mtau$  commute if and only if $e^a(s,\infty)e^b(s,\infty)=e^b(s,\infty)e^a(s,\infty)$ for all $s>0$.
\end{proposition}

The following facts will be used in the subsequent sections.
\begin{lemma}\cite[Lemma 2.5]{DDPnoncomm}, \cite{noncomm}
\label{lm:aux1}
Suppose that the von Neumann algebra $\M$ is non-atomic and $a\in S^+\Mtau$. If $\lambda<\tauone$ then there exists $e\in P(\M)$ such that
\[
e^a\left(\mu(\lambda,a),\infty)\right)\leq e\leq e^a\left[\mu(\lambda,a),\infty)\right)\quad\text{and}\quad\tau(e)=\lambda.
\]
\end{lemma}
\begin{lemma}\cite[Lemma 3.2]{DDPnoncomm}, \cite{noncomm}
\label{lm:aux2}
Suppose that  $a\in S^+\Mtau$, $e\in P(\M)$ is such that $\tau(e)<\infty$ and
\[
e^a\left(\lambda,\infty\right)\leq e\leq e^a[\lambda,\infty).
\]
Then $ae=ea=eae$ and $\mu(ae)=\mu(a)\chi_{[0,\tau(e))}$. In particular,
\[
\tau(eae)=\tau(ae)=\int_0^{\tau(e)}\mu(a).
\]
\end{lemma}
\begin{lemma}\cite{noncomm}
\label{lm:aux3}
Suppose that $0\leq x\in L_1\Mtau+\M$. Let $\lambda>0$ and $e$ be a projection in $\M$ such that $\tau(e)=\lambda$ and
\[
\tau(xe)=\int_0^{\lambda}\mu(x).
\]
Then
\[
e^x\left(\mu(\lambda,x),\infty\right)\leq e\leq e^x\left[\mu(\lambda,x),\infty\right).
\]
\end{lemma}

\begin{proposition}
\label{prop:singfun}
\begin{itemize}
\item[(1)] \cite[Proposition 1.1]{moj} For $x\in S\Mtau $, \[\mu(\abs{x}+\mu(\infty,x)n(x))=\mu(x)\].
\item[(2)]\cite[Proposition 1.1]{moj} If $x\in S\Mtau $ and $\abs{x}\geqslant\mu(\infty,x)s(x)$ then $\mu(\abs{x}-\mu(\infty,x)s(x))=\mu(x)-\mu(\infty,x)$.
\item[(3)]\cite{noncomm}, \cite[Lemma 1.3]{CKK1}  If $x,y\in L_1\Mtau+\mathcal{M}$, $\mu(\infty,x)=\mu(\infty,y)=0$ and $\mu\left((x+y)/2\right)=\mu(x)=\mu(y)$, then $x=y$. 
\item[(4)] \cite{SCh} If $x,y\in S\Mtau$, $x=x^*$, $y\geq 0$ and $-y\leq x\leq y$, then $x\prec y$.
\item[(5)] \cite[Proposition 2.2]{CKSextreme} If $x,y\in S^+(\mathcal{M},\tau)$, $y\neq 0$ and $x\geq \mu(\infty,x)\one$, then there exists $s>0$ such that $\mu(s,x)<\mu(s,x+y)$.
\item[(6)] \cite[Proposition 3.5]{CKSextreme} If $x,y\in S(\mathcal{M},\tau)$, $y=y^*$, $x\geq \mu(\infty,x)\mathbf{1}$ and $\mu(x+iy)=\mu(x)$, then $y=0$.
\item[(7)]\cite{noncomm} If $s\geq 0$ and $p=e^{\abs{x}}(s,\infty)$ then $\mu(\abs{x}p)=\mu(x)\chi_{[0,\tau(p))}$.
\end{itemize}
\end{proposition}
Condition (1) in the above proposition implies the following.
\begin{corollary}
\label{cor:1} Let $x\in S\Mtau$ and $p\in \mathcal{P}(\M)$. If $px=xp=0$ and $0\leq C\leq \mu(\infty,x)$ then $\mu(x+Cp)=\mu(x)$.
\end{corollary}
\begin{proof}
Suppose that $px=xp=0$ and $0\leq C\leq \mu(\infty,x)$. Then $n(x)\geq p$, $n(x^*)\geq p$ and $\abs{x+Cp}=\abs{x}+Cp$. The claim follows now by Proposition \ref{prop:singfun} (1). Indeed,   $\mu(\abs{x})\leq\mu(\abs{x}+Cp)=\mu(x+Cp)\leq \mu(\abs{x}+\mu(\infty,x)n(x))=\mu(x)$.
\end{proof}

Below we include the results that show that $E$ is isometrically embedded into $\nonsp$, if certain conditions on the trace $\tau$ and von Neumann algebra $\M$ are imposed.

Recall that given two $*$-algebras $\mathcal A$ and $\mathcal B$, the mapping $\Phi:\mathcal{A}\to\mathcal{B}$ is  called a \textit{$*$-homomorphism} if $\Phi$ is an algebra homomorphism and $\Phi(x^*)=\Phi(x)^*$ for all $x\in\mathcal{A}$. If, in addition, $\mathcal A$ and $\mathcal B$ are unital and $\Phi(\one_A)=\one_B$ then $\Phi$ is called \textit{unital $*$-homomorphism}.
\begin{proposition}\cite{noncomm}, \cite[Lemma 1.3]{CKSlur}
\label{isom2}
Let $\mathcal{M}$ be a non-atomic von Neumann algebra with a faithful, normal,  $\sigma$-finite trace $\tau$ and $x\in
S_0^+\left(\mathcal{M},\tau\right)$.  Then there exists a non-atomic commutative von Neumann subalgebra $\mathcal{N}$ in $ \M$  and a $*$-isomorphism $U$ from the $*$-algebra $S(\mathcal{N},\tau)$ onto the $*$-algebra $S\left(\left[0,\tauone\right),m\right)$ such that $x\in S(\mathcal{N},\tau)$ and $\mu(y)=\mu(Uy)$ for every $y\in S(\mathcal{N},\tau)$.
\end{proposition}

Given an operator $x\in S\Mtau$ and a projection $p\in \mathcal{P}(\M)$ we define the von Neumann algebra $\M_{p}=\{py|_{p(H)}:\quad y\in\M\}$. It is known that there is a unital  $*$-isomorphism from $S\left(\M_p,\tau_p\right)$ onto $pS\Mtau p$. Moreover, the decreasing rearrangement $\mu^{\tau_{p}}$ computed with respect to the von Neumann algebra $\left(\M_{p},\tau_{p}\right)$ is given by $\mu^{\tau_{p}}(y)=\mu(pyp)$, $y\in S\left(\M_p,\tau_p\right)$.
See \cite{moj,CKK1, DDPnoncomm} for details.

Using the theory of measure preserving transformations which retrieve functions from their decreasing rearrangements and $U^{-1}$ from  Proposition \ref{isom2} the following can be shown.

\begin{proposition} \cite{noncomm}
\label{isom1}
Suppose that $\mathcal{M}$ is a non-atomic von Neumann algebra with a faithful, normal  trace $\tau$. 
Let $x\in\left( L_1(\mathcal{M},\tau)+\mathcal{M}\right)\cap S_0^+\left(\mathcal{M},\tau\right)$. Then there exist a non-atomic
commutative von Neumann subalgebra \newline
$\mathcal{N}\subset s(x)\M s(x)$ and a unital $*$-isomorphism
$V$ acting from the $*$-algebra\linebreak $S\left(\left[0,\tau(s(x))\right),m\right)$ into the $*$-algebra $S(\mathcal{N},\tau)$, such that 
\[ V\mu(x)=x\ \ \ \text{ and }\ \ \ \mu(V(f))=\mu(f)\ \ \text{ for all } f\in S\left(\left[0,\tau(s(x))\right),m\right).
\]
\end{proposition}

\begin{proposition} \cite{noncomm}
\label{isom3}
Suppose that $\mathcal{M}$ is a non-atomic von Neumann algebra with a faithful, normal,  $\sigma$-finite trace $\tau$.
Let $x\in\left( L_1(\mathcal{M},\tau)+\mathcal{M}\right)\cap S_0^+\left(\mathcal{M},\tau\right)$ and $\tau(s(x))<\infty$. Then there exist a non-atomic commutative von Neumann subalgebra $\mathcal{N}\subset\M$ and a unital $*$-isomorphism $V$ acting from the $*$-algebra $S\left(\left[0,\tauone\right),m\right)$ into the $*$-algebra $S(\mathcal{N},\tau)$, such that 
\[ V\mu(x)=x\ \ \ \text{ and }\ \ \ \mu(V(f))=\mu(f)\ \ \text{ for all } f\in S\left([0,\tauone),m\right).
\]
\end{proposition}
We will need further a specific version of the above propositions.
\begin{corollary}
\label{cor:isom}
Let $\mathcal{M}$ be a non-atomic von Neumann algebra with a faithful, normal,  $\sigma$-finite trace $\tau$, $x\in S\Mtau$, and $\abs{x}\geq \mu(\infty,x)s(x)$. Denote by $p=s(\abs{x}-\mu(\infty,x)s(x))$ and define projection $q\in \mathcal{P}(\M)$ in the following way.
\begin{itemize}
\item[{(i)}] If $\tau(s(x))<\infty$ set $q=\one$.
\item[{(ii)}] If $\tau(s(x))=\infty$ and $\tau(p)<\infty$, set $q=s(x)$.
\item[{(iii)}] If $\tau(p)=\infty$, set $q=p$.
\end{itemize}
 Then there exist a non-atomic commutative von Neumann subalgebra $\mathcal{N}\subset q\M q$ and a unital $*$-isomorphism
$V$ acting from the $*$-algebra $S\left(\left[0,\tauone\right),m\right)$ into the $*$-algebra $S(\mathcal{N},\tau)$, such that 
\[ V\mu(x)=\abs{x}q\ \ \ \text{ and }\ \ \ \mu(V(f))=\mu(f).
\]
for all  $f\in S\left([0,\tauone),m\right)$.
\end{corollary}
\begin{proof}
Observe that $p=s(\abs{x}-\mu(\infty,x)s(x))=e^{\abs{x}}(\mu(\infty,x),\infty)\leq s(x)$.  Hence if $\tau(p)=\infty$ then also $\tau(s(x))=\infty$, and therefore conditions (i), (ii), and (iii) give all possible cases. Furthermore, by Proposition \ref{prop:singfun} (2), $\mu(\abs{x}-\mu(\infty,x)s(x))=\mu(x)-\mu(\infty,x)$, and so $\abs{x}-\mu(\infty,x)s(x)\in S_0^+\Mtau$. 

Note that in either case $\tau(q)=\tauone$. Hence in view of Proposition \ref{prop:singfun} (7),  it follows that $\mu(\abs{x}q)=\mu(x)\chi_{[0,\tau(q))}=\mu(x)$.

Case (i). Since $\tau(s(x))<\infty$, we have that $\mu(\infty,x)=0$. Therefore the claim is an immediate consequence of the Proposition \ref{isom3} applied to $\abs{x}$.

Case (ii). Let $\tau(s(x))=\infty$, $\tau(p)<\infty$ and $q=s(x)$. 
Applying Proposition \ref{isom3} to the operator $\abs{x}-\mu(\infty,x)s(x)=s(x)(\abs{x}-\mu(\infty,x)s(x))s(x)\in s(x)S\Mtau s(x)$ and to the von Neumann algebra $s(x)\M s(x)$, there exists a non-atomic commutative von Neumann algebra $\mathcal{N}\subset s(x)\M s(x)$ and a $*$-isomorphism $V$ from  $S\left(\left[0,\tau(s(x))\right),m\right)= S\left(\left[0,\infty\right),m\right)$ into $S(\mathcal{N},\tau)$ such that 
\[ 
V\mu(\abs{x}-\mu(\infty,x)s(x))
=\abs{x}-\mu(\infty,x)s(x)\,\text{ and }\, \mu(V(f))=\mu(f)
\]
for all  $f\in S\left([0,\infty),m\right)$.
Since $V(\chi_{[0,\infty)})=s(x)$,
\begin{align*}
\abs{x}-\mu(\infty,x)s(x)&=V\mu(\abs{x}-\mu(\infty,x)s(x))=V(\mu(x)-\mu(\infty,x))\\
&=V\mu(x)-\mu(\infty,x)V(\chi_{[0,\infty)})=V\mu(x)-\mu(\infty,x)s(x),
\end{align*}
and consequently $V\mu(x)=\abs{x}=\abs{x}s(x)$.

Case (iii). Assume that $\tau(p)=\infty$ and $q=p$. By Proposition \ref{isom1} applied to the operator $\abs{x}-\mu(\infty,x)s(x)$ and von Neumann algebra $\M$, there exists a non-atomic commutative von Neumann algebra $\mathcal{N}\subset p\M p$ and a $*$-isomorphism $V$ from  $S\left(\left[0,\tau(p)\right),m\right)= S\left(\left[0,\infty\right),m\right)$ into $S(\mathcal{N},\tau)$ such that 
\[ 
V\mu(\abs{x}-\mu(\infty,x)s(x))=\abs{x}-
\mu(\infty,x)s(x)\,\text{ and }\,\mu(V(f))=\mu(f)
\]
for all $f\in S\left([0,\infty),m\right)$.
Since  $p\leq s(x)$, \[\abs{x}-\mu(\infty,x)s(x)=(\abs{x}-\mu(\infty,x)s(x))p
=\abs{x}p-\mu(\infty,x)p\] and $V(\chi_{[0,\infty)})=p$. Thus  again we have
\begin{align*}
\abs{x}p-\mu(\infty,x)p&=\abs{x}-\mu(\infty,x)s(x)
=V\mu(\abs{x}-\mu(\infty,x)s(x))\\
&=V\left(\mu(x)-\mu(\infty,x)\right)
=V\mu(x)-\mu(\infty,x)V(\chi_{[0,\infty)})\\
&=V\mu(x)-\mu(\infty,x)p,
\end{align*}
and $V\mu(x)=\abs{x}p$.
\end{proof}
\begin{remark}
\label{rm2}
We will describe below the construction of a non-atomic von Neumann algebra $\mathcal{A}$ with the trace $\kappa$, such that $E\Mtau$  embeds isometrically into $E(\mathcal{A},\kappa)$, for any symmetric function space $E$.

  Let
$\mathcal{A}=\mathcal{N}\overline\otimes \mathcal{M}$ be a tensor
product of von Neumann algebras $\mathcal{N}$ and $\mathcal{M}$, where $\mathcal{N}$ is a commutative von Neumann algebra identified with $L_{\infty}[0,1]$ with the trace $\eta$.  Let 
$\kappa= \eta\otimes \tau$ be a tensor product of the traces $\eta$ and $\tau$, that is $\kappa(N_f\otimes x)=\eta(N_f)\tau(x)$, \cite{KR,Takesaki}. It is well known that
$\left(\mathcal{A},\kappa\right)$ has no atoms.

Let $\mathds{1}$ be the identity operator on $L^2[0,1]$ and denote by $\Complex\mathds{1}=\{\lambda \mathds{1}:\lambda\in\Complex\}$.
 Let $x\in
S\left(\mathcal{M},\tau\right)$ and consider a linear subspace $D$
in $ L_2[0,1]\otimes H$ generated by the vectors of the form
$\zeta\otimes\xi$, where  $\zeta\in
L_2[0,1]$ and $\xi\in D(x)\subset H$. For every $\displaystyle \alpha=\sum_{i=1}^n\zeta_i\otimes\xi_i
\in D$ define $\displaystyle
(\mathds{1}\otimes x)(\alpha)=\sum_{i=1}^n\zeta_i\otimes x\xi_i$.
The linear operator $\mathds{1}\otimes x:D\rightarrow L_2[0,1]\otimes
H $ with  domain $D$ is preclosed, and by Lemma 1.2 in \cite{CKSlur} its closure
$\mathds{1}\overline\otimes x$ is contained in
$S(\Complex\mathds{1}\otimes \M,\kappa)$.

The map $\pi:x\rightarrow \mathds{1}\otimes x$, $x\in\M$, is a unital trace preserving $*$-isomorphism from $\M$ onto the von Neumann subalgebra $\Complex\mathds{1}\otimes\M$. Consequently, $\pi$ extends uniquely to a $*$-isomorphism $\tilde{\pi}$ from $S\Mtau$ onto $S(\Complex\mathds{1}\otimes \M,\kappa)$ \cite{noncomm}.  In fact one can show that $\tilde{\pi}(x)=\mathds{1}\overline\otimes x$.

Since every $*$-homomorphism is an order preserving map, $x\geq 0$ if and only if $\mathds{1}\overline\otimes x\geq 0$, where $x\in S\Mtau$.
 The spectral measure $e^{\tilde{\pi}(x)}$ of $\tilde{\pi}(x)$ is given by $e^{\tilde{\pi}(x)}(B)=\pi (e^x(B))$, that is  $e^{\mathds{1}\overline\otimes x}(B)=\mathds{1}\otimes e^{x}(B)$ for any Borel set $B\subset \mathbb{R}$. Hence $\kappa\left(e^{\mathds{1}\overline\otimes x}(s,\infty)\right)=\kappa\left( \mathds{1}\otimes e^{x}(s,\infty)\right)=\tau(e^x(s,\infty))$ for any $s>0$. Consequently $\tilde{\pi}$ preserves the singular value function in the sense that $\tilde{\mu}(\mathds{1}\overline\otimes x)=\mu(x)$, where $\tilde{\mu}(\mathds{1}\overline\otimes x)$ is the singular value function of $\mathds{1}\overline\otimes x$ computed with respect to the von Neumann algebra $\Complex \mathds{1}\otimes\M$ and the trace $\kappa$.

Thus
\[
\|\tilde{\pi}(x)\|_{E(\Complex \mathds{1}\otimes\M,\kappa)}=\norme{\tilde{\mu}(\mathds{1}\overline\otimes x)}=\norme{\mu(x)}=\normcomm{x},
\]
where 
\begin{align*}
E(\Complex \mathds{1}\otimes\M,\kappa)&=\{\mathds{1}\overline\otimes x\in S(\Complex \mathds{1}\otimes\M,\kappa): \tilde{\mu}(\mathds{1}\overline\otimes x)\in E\}
\\&=\{\mathds{1}\overline\otimes x: x\in S\Mtau\text{ and }\mu(x)\in E\}.
\end{align*}

Hence $\tilde{\pi}$ is a $*$-isomorphism which is also an isometry from $\nonsp$ onto $E(\Complex \mathds{1}\otimes\M,\kappa)$.
 We refer reader to  \cite{CKSlur,noncomm,  St} for details.
\end{remark}

Given a natural number $k$, consider a normed linear space $X$ over the field $\Complex$ of complex numbers whose dimension is at least $k+1$. Denote by $S_X$ and $B_X$ the unit sphere and the unit ball of $X$, respectively.

\begin{definition}
A point $x\in S_X$ is called \textit{$k$-extreme} of the unit ball $B_X$ if $x$ cannot be represented as an average of $k+1$ linearly independent elements from the unit sphere $S_X$. Equivalently, $x$ is $k$-extreme whenever the condition $x=\frac{1}{(k+1)}\sum_{i=1}^{k+1}$,  $x_i\in S_X$ for $i=1,\dots,k+1$, implies that $x_1, x_2,...,x_{k+1}$ are linearly dependent.
Moreover, if every element of the unit sphere $S_X$ is $k$-extreme, then $X$ is called \textit{$k$-rotund}.
\end{definition}

The notion of $k$-extreme points was explicitly  introduced in \cite{ZYD} and applied to theorem on uniqueness of Hahn-Banach extensions. More precisely, Zheng and Ya-Dong  showed there that given at least $k+1$-dimensional normed linear space over the complex field, all bounded linear functionals defined on subspaces of $X$ have at most $k$-linearly
independent norm-preserving linear extensions to $X$ if and only if the conjugate space $X^*$ is $k$-rotund.
In the paper \cite{BFLM} $k$-rotundity and $k$-extreme points found interesting application in studying the structure of nested sequences of balls in Banach spaces.

Clearly, if $X$ is a normed space of a dimension at least $l$, where $l\geq k$, and $x\in S_X$ is a $k$-extreme point of $B_X$, then $x$ is $l$-extreme. Moreover, $1$-extreme points are just extreme points of $B_X$, and so $1$-rotundity of $X$ means rotundity of $X$. 

The simple example below differentiates between $k$-extreme and $k+1$-extreme points.
\begin{example}
Given $k\in\mathbb{N}$, consider the $k+2$ dimensional space $\ell_1^{k+2}$, equipped with $\ell_1$ norm. Let $x=\left(\frac1{k+1},\frac 1{k+1},\dots,\frac 1{k+1},0\right)$. Clearly $x$ is not a $k$-extreme point of $B_{\ell_1^{k+2}}$, since it can be written as an average of $k+1$ linearly independent unit vectors $e_1, e_2, \dots, e_{k+1}$. However, $x$ is $k+1$-extreme point of $B_{\ell_1^{k+2}}$. To see it, let $x=\frac 1{k+2}\sum_{i=1}^{k+2} y_i$,
where $y_i=\left(y_i^{(1)},y_i^{(2)},\dots,y_i^{(k+2)}\right)\in S_{\ell_1^{k+2}}$, $i=1,2,\dots, k+2$. Then
\begin{align*}
k+2&=(k+2)\|x\|_{1}=\sum_{j=1}^{k+2}\sum_{i=1}^{k+2} y_i^{(j)}=\sum_{i=1}^{k+2}\sum_{j=1}^{k+2} y_i^{(j)}\leq \sum_{i=1}^{k+2}\sum_{j=1}^{k+2} \abs{y_i^{(j)}}\\
&=\sum_{i=1}^{k+2}\|y_i\|_1=k+2.
\end{align*}
Therefore ${y_i^{(j)}}=\abs{y_i^{(j)}}$ for ever $i,j=1,2,\dots,k+2$. In particular $y_i^{(k+2)}\geq 0$, $i=1,2,\dots,k+2$, and in view of $\sum_{i=1}^{k+2} y_i^{(k+2)}=(k+2)x^{(k+2)}=0$ it follows that $y_i^{(k+2)}= 0$ for every $i=1,2,\dots,k+2$. Thus the matrix formed by vectors $y_1, y_2,\dots, y_{k+2}$ has determinant equal to zero, since the last row comprises only of zeros. Consequently, $y_1,y_2,\dots, y_{k+2}$ are linearly dependent, and $x$ is $k+1$-extreme.
\end{example}
 We wish to mention here that also the family of Orlicz sequence spaces exposes the difference between $k$-extreme and $k+1$-extreme points \cite{Chen}.

The structure of the paper is as follows.
Section 2 focuses on  $k$-extreme points in symmetric Banach function spaces. A new characterization of $k$-extreme points in a Banach space is given in  Proposition \ref{lm:2}. The main theorem of the section, Theorem \ref{thm:orbitmain}, is the analogous result to Ryff's description of extreme points of orbits \cite{Ryff}. Part 3 considers $k$-extreme points in the noncommutative space $\nonsp$. The main results of this section, Theorem \ref{thm:noncom4} and Theorem \ref{thm:noncom5}, characterize $k$-extreme operators in terms of their singular value functions. They generalize the result in \cite{CKSextreme} proved for extreme points in the case of $k=1$.
The closing corollary of the section relates $k$-rotundity of $E$ with the $k$-rotundity of $\nonsp$. In the last section of the paper we apply the obtained results to characterize $k$-extreme points of orbits $\Omega(g)$ and $\Omega'(g)$. Given $g\in L_1+L_{\infty}$ (resp. $g\in L_1[0,\alpha),\, \alpha<\infty$), we have that $f$ is a $k$-extreme point of its orbit $\Omega(g)$ (resp. $\Omega'(g)$) if and only if $\mu(f)=\mu(g)$ and $\abs{f}\geq \mu(\infty,f)$ (resp. $\mu(f)=\mu(g)$). Therefore we obtain that $k$-extreme points of orbits $\Omega(g)$, and consequently of unit balls of Marcinkiewicz spaces, are non-distinguishable from extreme points. 

\section{$k$-extreme points in symmetric function spaces}
Let us first state an equivalent definition of a $k$-extreme point in a normed space $X$. 
 We will need the following simple observation, included in Lemma 1 \cite{YD}.
\begin{lemma}
\label{lm:1}
If $x_1,x_2,\dots,x_n$ are linearly dependent in $B_X$ and $\|x_1+x_2+\dots+x_n\|=n$, then there are complex numbers $c_1,c_2,\dots,c_n$, not all zero, such that $c_1+c_2+\dots+c_n=0$ and $c_1x_1+c_2x_2+\dots+c_nx_n=0$.
\end{lemma}

\begin{proposition}
\label{lm:2}
An element $x\in S_X$ is $k$-extreme of $B_X$ if and only if whenever for the elements $u_i\in X$, $i=1,\dots, k$, the conditions $x+u_i\in B_X$ and $x-\sum_{i=1}^{k} u_i\in B_X$ imply that $u_1,u_2,\dots, u_k$ are linearly dependent.
\end{proposition}
\begin{proof}
Suppose that $x\in S_X$ is $k$-extreme and $x+u_i\in B_X$ and $x-\sum_{i=1}^{k} u_i\in B_X$, for $u_1,u_2,\dots, u_k\in X$. Set 
\[
y_i=x+u_i,\text{ for }i=1,\dots, k,\quad\text{and}\quad y_{k+1}=x-\sum_{i=1}^{k} u_i.
\]
By the assumption we have $y_i\in B_X$ for $i=1,2,\dots,k+1$, and $\sum_{i=1}^{k+1} y_i=(k+1) x$. Consequently, $y_i\in S_X$, $i=1,2,\dots,k+1$, and since $x$ is $k$-extreme, $y_1,y_2,\dots,y_{k+1}$ are linearly dependent. By Lemma \ref{lm:1}, there exist complex numbers $c_1,c_2,\dots,c_{k+1}$, not all equal to zero, such that $c_1+c_2+\dots+c_{k+1}=0$ and $c_1y_1+c_2y_2+\dots+c_{k+1}y_{k+1}=0$. Therefore  $(c_1-c_{k+1})u_1+(c_2-c_{k+1})u_2+\dots+(c_{k}-c_{k+1})u_{k}=0$ and $u_1,u_2, \dots,u_k$ are linearly dependent, since $c_i-c_{k+1}\neq 0$ for some $i=1,2,\dots,k$.


Suppose now that $x$ is not a $k$-extreme point of $B_X$, that is there exist linearly independent vectors $x_1, x_2, \dots,x_{k+1}$ from the unit sphere $S_X$, such that $x=\frac{1}{(k+1)}\sum_{i=1}^{k+1}x_i$. Define $u_i=x_{i+1}-x$, for $i=1,2,\dots,k$. Note that $u_i\neq 0$ for all $i=1,2,\dots,k$, since $\{x_1,x_2,\dots,x_{k+1}\}$ are linearly independent. Then
$x+u_i=x_{i+1}\in B_X$, $i=1,2,\dots,k$, and $x-\sum_{i=1}^ku_i=(k+1)x-\sum_{i=2}^{k+1} x_i=x_1\in B_X$. Moreover, it is not difficult to see that $u_1, u_2,\dots, u_k$ are linearly independent. Indeed, suppose that 
\[
\lambda_1u_1+\lambda_2u_2+\dots+\lambda_ku_k=0, \text{ for some } \lambda_1,\lambda_2,\dots,\lambda_k\in \Complex.
\]
We have then the following equivalent equalities
\begin{align*}
&\lambda_1x_2+\lambda_2x_3+\dots+\lambda_kx_{k+1}-
(\lambda_1+\lambda_2+\dots+\lambda_k)x=0,\\
&\lambda_1x_2+\dots+\lambda_kx_{k+1}-\frac{1}{k+1}(\lambda_1+\dots+\lambda_k)(x_1+\dots+x_{k+1})=0,\\
&-\frac{1}{k+1}(\lambda_1+\dots+\lambda_k)x_1-\frac{1}{k+1}(-k\lambda_1+\dots+\lambda_k)x_2-\dots\\
&-\frac{1}{k+1}(\lambda_1+\dots-k\lambda_k)x_{k+1}=0.
\end{align*}
Since $x_1,x_2,\dots,x_{k+1}$ are linearly independent it follows that $\lambda_1=\lambda_2=\dots=\lambda_k=0$, and so $u_1,u_2,\dots,u_k$ are linearly independent.
\end{proof}

Our main goal now is to show an analogous theorem on $k$-extreme points as the Ryff's theorem on extreme points \cite{Ryff}. We need first several preliminary results. Let us first note that $E$ stands in this section for a symmetric Banach function space on $[0,\tauone)$, where $\tauone\leq \infty$.
\begin{proposition}
\label{prop:orbit1}
Let $f$ and $g$ be decreasing, right continuous functions from the unit sphere $S_E$. Assume there exist points $0<s_1<s_2<s_3<s_4<\infty$ such that
\begin{align*}
f(s_i)>f(s_{i+1}),\quad &i=1,2,3,\\
\int_0^s g>\int_0^{s_1}f+f(s_1)(s-s_1),\quad&\text{for all }s_1\leq s\leq s_4.
\end{align*}
If $f\prec g$ then $f$ is not a $k$-extreme point, for any $k\in\mathbb{N}$.
\end{proposition}
\begin{proof}
Suppose that all assumptions of the proposition are satisfied, and let
\[
\epsilon=\frac{1}{k}\min\{f(s_1)-f(s_2),f(s_3)-f(s_4)\}\quad\text{ and }\quad\delta=\frac{1}{2}(s_3-s_2).
\]
Define 
\[
u_i=-\epsilon\chi_{(s_2,s_2+\frac{1}{k}\delta)}+\epsilon\chi_{(s_3-\frac{i}{k}\delta,s_3-\frac{i-1}{k}\delta)},\quad i=1,2,\dots,k.
\]
Set $f_i=f+u_i$, for $i=1,2,\dots,k$ and $f_{k+1}=f-\sum_{i=1}^k u_i$.

Let us show first that $f_{k+1}\prec g$.  Note that
  \[
f_{k+1}=f+k\epsilon\chi_{(s_2,s_2+\frac{1}{k}\delta)}-\epsilon\chi_{(s_3-\delta,s_3)},
\]
and so for $0\leq s\leq s_1$ and $s\geq s_4$, $\mu(s,f_{k+1})=f(s)$. Hence taking $s_1\leq s\leq s_4$ we have that $\mu(s,f_{k+1})\leq f(s_1)$ and
\[
\int_0^s \mu(f_{k+1})=\int_0^{s_1}f+\int_{s_1}^s\mu(f_{k+1})\leq\int_0^{s_1}f+f(s_1)(s-s_1)<\int_0^s g,
\]
\[
\int_{s_1}^{s_4}\mu(f_{k+1})=\int_{s_1}^{s_4}f_{k+1}=\int_{s_1}^{s_4}f+k\epsilon\cdot\frac{1}{k}\delta-\epsilon\delta=\int_{s_1}^{s_4} f.
\]
Combining the previous equality with the fact  $\mu(s,f_{k+1})=f(s)$ for $s\geq s_4$ and $0<s\leq s_1$, it follows that for $s\geq s_4$,
\begin{align*}
\int_0^s \mu(f_{k+1})&=\int_0^{s_1}\mu(f_{k+1})+\int_{s_1}^{s_4}\mu(f_{k+1})
+\int_{s_4}^s\mu(f_{k+1})\\
&=\int_0^{s_4}f+\int_{s_4}^sf=\int_0^s f.
\end{align*}
Therefore $f_{k+1}\prec g$ and consequently $f_{k+1}\in B_E$. Similarly one can show that $f_i\prec g$ and so $f_i\in B_E$ for all $i=1,2,\dots,k$.

Since $u_1,u_2,\dots,u_k$ are linearly independent, by Proposition \ref{lm:2}, $f$ cannot be a $k$-extreme point.
\end{proof}

\begin{lemma}
\label{lm:orbit2}
Let $f$ and $g$ be decreasing, right continuous functions from the unit sphere $S_E$, such that $f\prec g$. If for some $t_0>0$, $f$ is not constant in any of its right neighborhoods, and
\[
\int_0^{t_0}f<\int_0^{t_0} g,
\]
then $f$ is not a $k$-extreme point of $B_E$.
\end{lemma}
\begin{proof}
Let $\eta>0$ be such that
\[
\int_0^{t_0}f+\eta<\int_0^{t_0} g.
\]
Since $f$ has infinitely many values on every right neighborhood of $t_0$, we can chose $t_0<s_1<s_2<s_3<s_4<t_0+\eta/(2f(t_0))$ such that $f(s_1)>f(s_2)>f(s_3)>f(s_4)$.

Since $s_1-t_0<\eta/(2f(t_0))$, taking $s\geq t_0$ we have that
\begin{align*}
\int_0^{s_1} f+\eta&=\int_0^{t_0}f+\eta+\int_{t_0}^{s_1}f
<\int_0^{t_0}g+\int_{t_0}^{s_1}f\\
&\leq \int_0^{t_0}g+f(t_0)(s_1-t_0)\leq \int_0^{t_0}g+\frac{\eta}{2}\leq\int_0^{s}g+\frac{\eta}{2}.
\end{align*}
Consequently for $s\geq t_0$,
\[
\int_0^{s_1}f+\frac{\eta}{2}<\int_0^s g.
\]
 Notice that
\[
0<s_4-s_1\leq\frac{\eta}{2f(t_0)}\leq \frac{\eta}{2f(s_1)}.
\]
Hence for $s_1\leq s\leq s_4$,
\[
\int_0^{s_1}f+f(s_1)(s-s_1)\leq\int_0^{s_1}f+f(s_1)(s_4-s_1)\leq\int_0^{s_1}f+\frac{\eta}{2}<\int_0^s g,
\]
and by Proposition \ref{prop:orbit1}, $f$ cannot be a $k$-extreme point of $B_E$.
\end{proof}

\begin{lemma}
\label{lm:orbit2v2}
Let $f$ and $g$ be decreasing, right continuous functions from the unit sphere $S_E$, such that $f\prec g$. If for some $t_0>0$, $f$ is continuous at $t_0$, not constant in any of its left neighborhoods  and
\[
f(t_0)<g(t_0),
\]
then $f$ is not a $k$-extreme point of $B_E$.
\end{lemma}
\begin{proof}
By continuity of $f$ at $t_0$ and the assumption that $f(t_0)<g(t_0)$ there exist $\delta,\eta_1>0$ such that $f(t)+\eta_1<g(t)$ for all $t\in[t_0-\delta,t_0]$.
Hence $\int_0^{t_0}f+\eta<\int_0^{t_0}g$, for some $\eta>0$.

Set $C=\min\{\delta,\eta/(2g(t_0-\delta))\}$. Note that $g(t_0-\delta)>0$, and so $C$ is well defined.

The assumptions on the function $f$ ensure that $f$ has infinitely many values on every left neighborhood of $t_0$.  Thus we can find $t_0-C<s_1<s_2<s_3<s_4<t_0$, such that $f(s_1)>f(s_2)>f(s_3)>f(s_4)$.
Note that $t_0-s_1<C$, and so  $t_0-s_1<\eta/(2g(t_0-\delta))$ and $t_0-\delta<s_1$. Thus for all $s\in[s_1,t_0]$,
\begin{align*}
\int_0^{s_1}f+\eta&\leq \int_0^{t_0} f + \eta<\int_0^{t_0}g=\int_0^{s_1}g+\int_{s_1}^{t_0}g\leq \int_0^{s_1}g+g(s_1)(t_0-s_1)\\&\leq\int_0^{s_1}g+g(t_0-\delta)\frac{\eta}{2g(t_0-\delta)}<\int_0^{s_1}g+\frac{\eta}{2}<\int_0^sg +\frac{\eta}{2}.
\end{align*}
Moreover  $0<s_4-s_1<\eta/(2g(t_0-\delta))\leq \eta/(2f(t_0-\delta))\leq\eta/(2f(s_1))$.
Consequently, for all $s_1\leq s\leq t_0$, we have that
\[
\int_0^{s_1}f+f(s_1)(s-s_1)\leq \int_0^{s_1}f+f(s_1)(s_4-s_1)\leq \int_0^{s_1}f+\frac{\eta}{2}<\int_0^s g.
\]
It follows now by Proposition \ref{prop:orbit1} that $f$ is not a $k$-extreme point of $B_E$.

\end{proof}
\begin{theorem}
\label{thm:orbitmain}
Let $E$ be a symmetric Banach function space and $f\in S_E$. Suppose there exists a function $g\in S_E$ such that  $f\prec g$ and $\mu(f)\neq \mu(g)$. Then $\mu(f)$ cannot be a $k$-extreme point of $B_E$.
\end{theorem}
\begin{proof}
Since the condition $f\prec g$ is equivalent with $\mu(f)\prec \mu(g)$, we can assume that $f=\mu(f)$ and $g=\mu(g)$. Thus $f$ and $g$ are decreasing, right continuous functions. If $f\neq g$, then for at least one value of $t>0$,
\[
\int_0^t f<\int_0^t g.
\]
Then the set $A=\{u>0:\,f(u)<g(u)\}$ contains a point at which $f$ is continuous. Indeed, the condition that $f\neq g$ implies that $A$ is non-empty.  Furthermore, since $g$ is right continuous and $f$ is decreasing, one can show that the set $A$ contains a non-empty interval. Since a decreasing function $f$ has only countably many points of discontinuity, the claim follows. 

  Let $t_0>0$ be such that $f(t_0)<g(t_0)$ and $f$ be continuous at $t_0$. 

\textbf{Case I}
Assume that $t_0\in(t_1,t_2)$, $0\leq t_1<t_2<\infty$, $f=c$ on $[t_1,t_2)$, and $f$ experiences a jump discontinuities greater than $\eta$ at $t_1$ and $t_2$. If $t_1=0$, disregard the discontinuity at $t_1$.
Note that since $f(t_0)<g(t_0)$, we have that $c=f(t)< g(t)$ for all $t\in[t_1,t_0]$.
Let $\alpha=\sup\{t_1\leq t\leq t_2:\,f(t)<g(t)\}.$ Clearly $\alpha>t_0$.

\textbf{ Case I.a} Assume that $\alpha=t_2$ or $g(t_2^-)\geq c$.  Let 
\[
\delta<\frac{t_0-t_1}{2k}\quad\text{and}\quad\epsilon<\frac{1}{k}\min\{g(t_0)-c,\eta\}.
\]
For $i=1,2,\dots,k$, define
\begin{align*}
 u_i&=-\epsilon\chi_{(t_1,t_1+\delta)}+\epsilon\chi_{(t_2-i\delta,t_2-(i-1)\delta)}\\
f_i&=f+u_i\\
f_{k+1}&=f-\sum_{i=1}^k u_i=f+k\epsilon\chi_{(t_1,t_1+\delta)}-\epsilon\chi_{(t_2-k\delta,t_2)}.
\end{align*}
It is clear that $\mu(f_{k+1})=f_{k+1}$. We will show next that  $f_{k+1}\prec g$. For $t\leq t_1$, $\int_0^t f_{k+1}=\int_0^t f\leq \int_0^t g$. Observe that $c=f\chi_{[t_1,t_2)}\leq g\chi_{[t_1,t_2)}$, and so $f_{k+1}\chi_{[t_1,t_2)}\leq g\chi_{[t_1,t_2)}$. Thus taking $t_1<t\leq t_2$,  it follows that $\int_0^t f_{k+1}=\int_0^{t_1} f+\int_{t_1}^t f_{k+1}\leq\int_0^{t_1} f+\int_{t_1}^t g\leq \int_0^t g$. Finally if $t>t_2$, $\int_0^t f_{k+1}=\int_0^t f+k\epsilon\delta-k\epsilon\delta=\int_0^t f\leq \int_0^t g$. 

Note next that $\mu(f_i)=f+\epsilon\chi_{(t_1,t_1+\delta)}-\epsilon\chi_{(t_2-\delta,t_2)}$, for all $i=1,2,\dots,k$. Following the same argument as above, one can show that $f_i\prec g$, $i=1,2,\dots,k$. By linear independence of $u_1,u_2,\dots,u_k$ and by Proposition \ref{lm:2}, $f$ is not a $k$-extreme point of $B_E$.

\textbf{Case I.b} Let now $\alpha=\sup\{t_1\leq t\leq t_2:\,f(t)<g(t)\}<t_2$ and $g(t_2^-)<c$. Then for $\gamma,\beta$ sufficiently small $g(t)<c-\beta$, whenever $t_2-\gamma\leq t<t_2$. Set
\[
\delta<\frac{1}{k}\min\left\{\frac{t_0-t_1}{2},t_2-\alpha,\alpha-t_1,\gamma\right\}\quad\text{and}\quad\epsilon<\frac{1}{k}\min\{g(t_0)-c,\eta,\beta\}.
\]
Define functions $u_i$ and $f_i$ as in Case I.a.

We will show first that $f_{k+1}\prec g$, where obviously $\mu(f_{k+1})=f_{k+1}$. Since for $t\leq t_1$, $f_{k+1}(t)=f(t)$ and for $t_1<t< \alpha$, $f_{k+1}(t)\leq g(t)$, it remains to consider the case of $t\geq\alpha$.
Let $\alpha\leq t\leq t_2$. By $g\chi_{[\alpha,t_2]}\leq f_{k+1}\chi_{[\alpha,t_2]}$ and $\int_0^{t_2}f_{k+1}=\int_0^{t_2}f$, it follows that 
\[
\int_0^tf_{k+1}=\int_0^{t_2}f-\int_t^{t_2}f_{k+1}\leq \int_0^{t_2}f-\int_t^{t_2}g\leq\int_0^{t_2}g-\int_t^{t_2}g=\int_0^t g.
\]
For $t>t_2$, $\int_0^t f_{k+1}=\int_0^t f\leq \int_0^t g$.

Since  $\mu(f_i)=f+\epsilon\chi_{(t_1,t_1+\delta)}-\epsilon\chi_{(t_2-\delta,t_2)}$, for all $i=1,2,\dots,k$, following the similar argument as above it is not difficult to observe that $f_i\prec g$, $i=1,2,\dots,k$. Again by Proposition \ref{lm:2}, $f$ is not a $k$-extreme point of $B_E$.

\textbf{Case II} Suppose now that $t_0\in(t_1,t_2)$, $0<t_1<t_2\leq\infty$, $f=c$ on $[t_1,t_2)$, and $f$ is continuous at $t_1$, where $t_1=\inf\{t:\, f(t)=c\}$. Since $f(t_0)<g(t_0)$, we have that $f(t_1)<g(t_1)$.  By Lemma \ref{lm:orbit2v2} applied to $t_1$ we can conclude now that $f$ is not a $k$-extreme point of $B_E$.

\textbf{Case III} Assume that $t_0\in(t_1,t_2)$, $0\leq t_1<t_2<\infty$, $f=c$ on $[t_1,t_2)$, and $f$ is continuous at $t_2$, where $t_2=\sup\{t:\,f(t)=c\}$. We claim that $\int_0^{t_2}f<\int_0^{t_2} g$ and so by Lemma \ref{lm:orbit2}, $f$ is not $k$-extreme.
Indeed, suppose that $\int_0^{t_2}f=\int_0^{t_2} g$. 

If $g(t_2^-)\geq f(t_2)$, by  the continuity of $f$ at $t_0$, and inequality $f(t_0)<g(t_0)$, we have $\int_{t_0}^{t_2}f<\int_{t_0}^{t_2}g$. Hence
\[
\int_0^{t_0} f=\int_0^{t_2}f-\int_{t_0}^{t_2}f=\int_0^{t_2}g-\int_{t_0}^{t_2}f>\int_0^{t_2}g-\int_{t_0}^{t_2}g=\int_0^{t_0}g,
\]
contradicting the assumption that $f\prec g$.

On the other hand if $g(t_2^-)< f(t_2)$, there exists $\delta>0$ such that $g(t)<f(t)$ for   $t\in(t_2,t_2+\delta)$. Then for $t\in(t_2,t_2+\delta)$,
\[
\int_0^t f=\int_0^{t_2}f+\int_{t_2}^t f=\int_0^{t_2}g+\int_{t_2}^t f>\int_0^{t_2}g+\int_{t_2}^t g=\int_0^{t}g,
\]
which leads to a contradiction. Hence $\int_0^{t_2} f<\int_0^{t_2} g$, where $f$ is not constant in any of  the right neighborhoods of $t_2$. By Lemma \ref{lm:orbit2} it follows that $f$ is not a $k$-extreme point.

\textbf{Case IV} Suppose now that $t_0\in(t_1,\infty)$, where $f=c$ on $[t_1,\infty)$, and $t_1=\inf\{t:\,f(t)=c\}\geq 0$. Moreover, assume that $f$ has a jump discontinuity greater than $\eta$ at $t_1$, if $t_1>0$.
Note first that for $t\geq t_1$,
\begin{equation}
\label{eq:orbit1}
\int_0^{t_1}f+c(t-t_1)=\int_0^t f\leq \int_0^t g,
\end{equation}
and so
\begin{equation}
\label{eq:orbit2}
\frac{\int_0^{t_1}f}{t-t_1}+c\leq\frac{\int_0^t g}{t-t_1}.
\end{equation}
It is clear that $\lim_{t\to\infty} \frac{\int_0^{t_1}f}{t-t_1}=0$. Moreover, by (\ref{eq:orbit1}) it is easy to see that $\lim_{t\to\infty}\int_0^t g=\infty$. Hence $\displaystyle \lim_{t\to\infty}\frac{\int_0^t g}{t-t_1}=\lim_{t\to\infty}g(t)$. Consequently by (\ref{eq:orbit2}), $\lim_{t\to\infty}g(t)\geq c$.
 Set
\[
\epsilon<\min\{\eta, g(t_0)-f(t_0)\}, \quad\text{and}\quad \delta=\frac{t_0-t_1}{k}.
\]  
If $t_1=0$ disregard $\eta$ in the inequality above.
Define the functions\\ $u_i=-\epsilon\chi_{(t_0-i\delta,t_0-(i-1)\delta)}$, $f_i= f+u_i$,  for $i=1,2,\dots,k$, and $f_{k+1}=f-\sum_{i=1}^ku_i$.

Consider first the decreasing function $f_{k+1}=f+\epsilon\chi_{(t_1,t_0)}$. We have that $f_{k+1}(t)=f(t)+\epsilon\chi_{(t_1,t_0)}(t)\leq g(t)$ for all $t\geq t_1$, and $f_{k+1}(t)=f(t)$ for $0<t<t_1$. It is easy to observe now that $f_{k+1}\prec g$.

 Moreover, for $i=1,2,\dots, k$, $\mu(f_i)=f$. In view of Proposition \ref{lm:2}, the claim follows.

\textbf{Case V} Consider now the case when $f$ differs from a constant in every neighbourhood of $t_0$. Since $f$ is continuous at $t_0$ and $f(t_0)<g(t_0)$, Lemma \ref{lm:orbit2v2} ensures now that $f$ is not a $k$-extreme point.
\end{proof}
\begin{corollary}
\label{cor:orbitmain}
 Let $E$ be a symmetric Banach function space and $f\in S_E$. If $\mu(f)$ is a $k$-extreme point of $B_E$ then for all functions $g\in S_E$ with $f\prec g$, it holds that $\mu(f)=\mu(g)$.
\end{corollary}

We wish to observe that the same characterization of the $k$-extreme points is not valid for symmetric sequence spaces. We are grateful to Timur Oikhberg for bringing it to our attention and providing the following example.

Consider the points $x=(\frac12,\frac12,0)$ and $y=(1,0,0)$ in $\ell_1$. It is easy to verify that $x$ is a $2$-extreme point in $\ell_1$ with $x\prec y$. However $x=\mu(x)\neq \mu(y)=y$.

\section{$k$-extreme points in noncommutative Spaces}

In this section we shall characterize $k$-extreme points of the ball $B_{\nonsp}$ in terms of their singular value functions. Through the effort of the series of technical lemmas we will establish main Theorems \ref{thm:noncom4} and \ref{thm:noncom5}.
\begin{lemma}
\label{lm:noncom0}
Let $\M$ be non-atomic and $x$ be a non-zero, positive element of $S\Mtau$. Then there exist $k$ mutually orthogonal, non-zero projections $p_1,p_2,\dots,p_k$, commuting with $x$ and such that $p_i\leq s(x)$, $i=1,2,\dots,k$.
\end{lemma}
\begin{proof}
Suppose first that $\mu(x)$ admits at least $k+1$ different values. Choose $0<\lambda_1<\lambda_2<\dots<\lambda_{k+1}$ such that $\mu(\lambda_1,x)>\mu(\lambda_2,x)>\dots>\mu(\lambda_{k+1},x)\geq\mu(\infty,x)$. Since $\tau(e^x(\lambda,\infty))<\infty$ for all $\lambda> \mu(\infty,x)$, we have that
\[
\tau\left(e^x\left(\mu(\lambda_{i+1},x),\mu(\lambda_{i},x)\right]\right)=m\{t:\,\mu(\lambda_{i+1},x)<\mu(t,x)\leq \mu(\lambda_i,x)\}>0,
\]
$i=1,2,\dots, k$.
Hence $0\neq p_i=e^x\left(\mu(\lambda_{i+1},x),\mu(\lambda_{i},x)\right]\leq e^x(0,\infty)=s(x)$, $i=1,2,\dots,k$. Clearly $p_1,p_2,\dots,p_k$ are mutually orthogonal projections commuting with $x$.

Assume now that $\mu(x)$ has less than $k+1$ different values. Then $\mu(0,x)<\infty$ is the biggest value of $\mu(x)$. Moreover $\tau\left(e^x[\mu(0,x),\infty)\right)=m\{t:\,\mu(t,x)\geq\mu(0,x)\}>0$, if $\mu(0,x)>\mu(\infty,x)$ or $e^x[\mu(0,x),\infty)=\one$, if $\mu(0,x)=\mu(\infty,x)$. In either case $e^x[\mu(0,x),\infty)$ is a non-zero projection less than $s(x)=e^x(0,\infty)$.

Since $\M$ is non-atomic, we can find $k$ mutually orthogonal, non-zero projections $p_1,p_2,\dots,p_{k}$ such that $p_i\leq e^x[\mu(0,x),\infty)$, $i=1,2,\dots,k$. We claim that $p_i$, $i=1,2,\dots,k$, commute with all spectral projections of the form $e^x(s,\infty)$, $s>0$. Indeed, if $s\geq\mu(0,x)$, then $e^x(s,\infty)=0$. For $s<\mu(0,x)$, $e^x[\mu(0,x),\infty)\leq e^x(s,\infty)$ and so $p_i\leq e^x(s,\infty)$, $i=1,2,\dots,k$. Thus if $s<\mu(0,x)$, $p_ie^x(s,\infty)=p_i=e^x(s,\infty)p_i$, $i=1,2,\dots,k$. Proposition \ref{prop:commute} implies now that all projections $p_1,p_2,\dots,p_k$ commute with $x$.
\end{proof}

\begin{lemma}
\label{lm:noncom2}
Let $\M$ be a non-atomic von Neumann algebra. If $x$ is a $k$-extreme point of the unit ball $\emph{B}_{\nonsp}$ then   $\mu(\infty,x)=0$ or $n(x)\mathcal{M}n(x^*)=0$.
\end{lemma}
\begin{proof}
Assume for a contrary that $n(x)\mathcal{M}n(x^*)\neq 0$ and
$\mu(\infty,x)>0$, while $x$ is a $k$-extreme point. As shown in Lemma 2.6 \cite{moj}, if $n(x)\mathcal{M}n(x^*)\neq 0$, there exists an isometry $w$ such that $x=w\abs{x}$.

It is easy to show now that if $x$ is $k$-extreme then so is $\abs{x}$. Indeed, let
$\abs{x}+u_i\in B_{\nonsp}$, for $i=1,2,\dots,k$ and $\abs{x}-\sum_{i=1}^k u_i\in B_{\nonsp}$.  Then $x+wu_i\in B_{\nonsp}$, $i=1,2,\dots,k$ and $x-\sum_{i=1}^k wu_i\in B_{\nonsp}$, and since $x$ is $k$-extreme $wu_1,wu_2,\dots, wu_k$ are linearly dependent. In view of  $w$ being an isometry, $u_1,u_2,\dots, u_k$ are linearly dependent, proving that $\abs{x}$ is $k$-extreme.

Since $n(x)\neq 0$ and $\M$ is non-atomic, there exist $k$ mutually orthogonal, non-zero projections, $p_1,p_2,\dots, p_k$, such that $p_i\leq n(x)$, $i=1,2,\dots, k$.
By Corollary \ref{cor:1},  $\mu(\abs{x}+\mu(\infty,x)p_i)=\mu(x)$, and $\mu(\abs{x}-\sum_{i=1}^k \mu(\infty,x)p_i)=\mu(\abs{x}+\sum_{i=1}^k \mu(\infty,x)p_i)=\mu(x)$. Clearly the set 
\[\{\mu(\infty,x)p_1,\mu(\infty,x)p_2,\dots,\mu(\infty,x)p_k\}
\] is linearly independent,  which is impossible since  $\abs{x}$ is $k$-extreme.
\end{proof}
\begin{lemma}
\label{lm:noncom3}
Let $\M$ be a non-atomic von Neumann algebra. If $x$ is a $k$-extreme point of $B_{\nonsp}$ then $\abs{x}\geq \mu(\infty,x)s(x)$.
\end{lemma}
\begin{proof}
Suppose that $\mu(\infty,x)>0$ and $e^{\abs{x}}(0,\mu(\infty,x))\neq 0$. We have that $\abs{x}e^{\abs{x}}(0,\mu(\infty,x))\neq 0$, since $e^{\abs{x}}(0,\mu(\infty,x))\leq s(x)=e^{\abs{x}}(0,\infty)$. 

Choose $0<\epsilon<1$ such that $e^{\abs{x}}(0,\beta]\neq 0$, where $\beta=\frac{1}{1+\epsilon}\mu(\infty,x)$. Such $\epsilon$ must exist, since  otherwise $e^{\abs{x}}(0,\mu(\infty,x))=0$. 

By Lemma \ref{lm:noncom0} applied to $\abs{x}e^{\abs x}(0,\beta]$, we can find $k$ non-zero, mutually orthogonal projections $\{p_1,p_2,\dots,p_k\}$, $p_i\leq e^{\abs{x}}(0,\beta]$ and commuting with $\abs{x}e^{\abs{x}}(0,\beta]$. Consequently $p_i$ commute with $\abs{x}$ for all $i=1,2,\dots, k$. Indeed, since $p_ie^{\abs{x}}(\beta,\infty)=e^{\abs{x}}(\beta,\infty)p_i=0$ and $p_i$ commutes with $\abs{x}e^{\abs{x}}[0,\beta]$, we have that for all $i=1,2,\dots, k$,
\begin{align*}
\abs{x}p_i&=\abs{x}e^{\abs{x}}\{0\}p_i+\abs{x}e^{\abs{x}}(0,\beta]p_i+\abs{x}e^{\abs{x}}(\beta,\infty)p_i=\abs{x}e^{\abs{x}}(0,\beta]p_i\\
&=p_i\abs{x}e^{\abs{x}}\{0\}+p_i\abs{x}e^{\abs{x}}(0,\beta]+p_ie^{\abs{x}}(\beta,\infty)\abs{x}\\
&=p_i\abs{x}e^{\abs{x}}\{0\}+p_i\abs{x}e^{\abs{x}}(0,\beta]+p_i\abs{x}e^{\abs{x}}(\beta,\infty)=p_i\abs{x}.
\end{align*}

 Define $u_i=-\epsilon \abs{x}p_i$,  $z_i=\abs{x}+u_i=\abs{x}-\epsilon \abs{x}p_i$, $i=1,2,\dots,k$, and $z_{k+1}=\abs{x}-\sum_{i=1}^ku_i=\abs{x}+\epsilon\abs{x}\sum_{i=1}^kp_i$. For $i=1,2,\dots,k$, $0\leq \abs{x}-\abs{x}p_i\leq z_i\leq\abs{x}$, and so $z_i\in B_{\nonsp}$. Furthermore, it was shown in Lemma 3.8 \cite{moj}, that $\abs{x}+\epsilon\abs{x}e^{\abs{x}}(0,\beta]\in B_{\nonsp}$.  Since $z_{k+1}\leq \abs{x}+\epsilon\abs{x}e^{\abs{x}}(0,\beta]$ it follows that also $z_{k+1}\in B_{\nonsp}$.

Let $x=u\abs{x}$ be the polar decomposition of $x$.  Therefore $x-\epsilon xp_i=uz_i\in B_{\nonsp}$, for $i=1,2,\dots,k$ and $x+\epsilon x\sum_{i=1}^k p_i=uz_{k+1}\in B_{\nonsp}$. Note that for any non-zero projection $q\leq e^{\abs{x}}(0,\beta]\leq s(x)$ we have that $xq\neq 0$. Therefore the collection $\{-\epsilon xp_i\}_{i=1}^k$ is linearly independent, and $x$ cannot be a $k$-extreme point.

Therefore $e^{\abs{x}}(0,\beta]=0$ for all $\beta<\mu(\infty,x)$. Hence $e^{\abs{x}}(0,\mu(\infty,x))=0$ and $s(x)= e^{\abs{x}}(0,\infty)=e^{\abs{x}}[\mu(\infty,x),\infty)$. Consequently,
\[
\abs{x}=\int_{[\mu(\infty,x),\infty)}\lambda de^{\abs{x}}(\lambda)\geq\mu(\infty,x)e^{\abs{x}}[\mu(\infty,x),\infty)=\mu(\infty,x)s(x).
\]

\end{proof}
The inequality  $\abs{x}\geq\mu(\infty,x)s(x)$ can be expressed in an equivalent way as follows.

\begin{lemma}
\label{lm:cond}
Let $x\in S\Mtau$ and $\mu(\infty,x)>0$. Then the conditions  $\abs{x}\geq\mu(\infty,x)s(x)$ and $\displaystyle e^{\abs{x}}(0,\mu(\infty,x))=0$ are equivalent.
\end{lemma}
\begin{proof}
It follows from the last paragraph of the proof of Lemma \ref{lm:noncom3} that if $e^{\abs{x}}(0,\mu(\infty,x))=0$ then $\abs{x}\geq\mu(\infty,x)s(x)$.  

For the converse, assume that $\abs{x}\geq\mu(\infty,x)s(x)$. Since for  any $\alpha>0$, $e^{\abs{x}}(0,\alpha]\leq s(x)$ and $\abs{x}$ commutes with $e^{\abs{x}}(0,\alpha]$, we have that $\abs{x}e^{\abs{x}}(0,\alpha]\geq\mu(\infty,x)e^{\abs{x}}(0,\alpha]$. Suppose that there exists $0<\alpha<\mu(\infty,x)$ such that $e^{\abs{x}}(0,\alpha]\neq 0$. Then 
\begin{align*}
\abs{x}e^{\abs{x}}(0,\alpha]&=\int_{(0,\alpha]}\lambda de^{\abs{x}}(\lambda)\leq \alpha\int_{(0,\alpha]} de^{\abs{x}}(\lambda)=\alpha e^{\abs{x}}(0,\alpha]\\
&\leq \mu(\infty,x)e^{\abs{x}}(0,\alpha],
\end{align*}
 and $\alpha e^{\abs{x}}(0,\alpha]\neq\mu(\infty,x)e^{\abs{x}}(0,\alpha]$. Hence for all $\alpha<\mu(\infty,x)$, $e^{\abs{x}}(0,\alpha]=0$.
Therefore $e^{\abs{x}}(0,\mu(\infty,x))=\vee_{0<\alpha<\mu(\infty,x)}e^{\abs{x}}(0,\alpha]=0$.
\end{proof}

\begin{theorem}
\label{thm:noncom4}
Suppose that $\M$ is a non-atomic von Neumann algebra with a $\sigma$-finite trace $\tau$.  If $x$ is a $k$-extreme point of $B_{\nonsp}$ then $\mu(x)$ is a $k$-extreme point of $B_E$ and either
\par {\rm(i)} $\mu(\infty,x)=0$ or
\par {\rm(ii)} $n(x)\mathcal{M}n(x^*)=0$ and $\abs{x}\geq\mu(\infty,x)s(x)$.
\end{theorem}
\begin{proof}
Suppose that $x$ is a $k$-extreme point of the unit ball in $E(\mathcal{M},\tau)$. By Lemma \ref{lm:noncom2} and \ref{lm:noncom3} conditions (i) or (ii) are satisfied. 

Let 
\begin{equation}
\label{eq:noncom4}
\mu(x)=\frac{1}{k+1}\sum_{i=1}^{k+1} f_i, \text{ where }f_i\in S_E, \,i=1,2,\dots, k+1.
\end{equation}
To prove that $\mu(x)$ is $k$-extreme we need to show that $f_1,f_2,\dots,f_{k+1}$ are linearly dependent.
Let 
\[
p=s(\abs{x}-\mu(\infty,x)s(x))=e^{\abs{x}}(\mu(\infty,x),\infty).
\]
By Corollary \ref{cor:isom}, there exist a projection $q\in \mathcal{P}(\M)$, a non-atomic commutative von Neumann subalgebra $\mathcal{N}\subset q\M q$ and a $*$-isomorphism $V$ acting from the $*$-algebra $S\left(\left[0,\tauone\right),m\right)$ into the $*$-algebra $S(\mathcal{N},\tau)$, such that 
\[ V\mu(x)=\abs{x}q\ \ \ \text{ and }\ \ \ \mu(V(f))=\mu(f)\ \ \text{ for all } f\in S\left([0,\tauone),m\right).
\]
Moreover, there are three choices of $q$: (1) $q=\one$ whenever $\tau(s(x))<\infty$, (2) $q=s(x)$ if $\tau(s(x))=\infty$ and $\tau(p)<\infty$, or (3) $q=p$ if $\tau(p)=\infty$.

Applying now isomorphism $V$ to the equation (\ref{eq:noncom4}),
\begin{equation}
\label{eq:noncom4.1}
\abs{x}q=V\mu(x)=\frac1{k+1}\sum_{i=1}^{k+1}V(f_i).
\end{equation}

Case (1). Let $\tau(s(x))<\infty$ and $q=\one$. Since $s(x)\sim s(x^*)$ and $\tau(s(x))<\infty$,  by \cite[Chapter 5, Proposition 1.38]{Takesaki} $n(x)\sim n(x^*)$. Then by \cite[Lemma 2.6]{moj} there exists an isometry $w$ such that $x=w\abs{x}$. 
Therefore by (\ref{eq:noncom4.1}) we have 
\[
x=\frac1{k+1}\sum_{i=1}^{k+1}wV(f_i),
\]
and $wV(f_1),wV(f_2),\dots, wV(f_{k+1})$ are linearly dependent. Since $w$ and $V$ are isometries $f_1, f_2, \dots, f_{k+1}$ are linearly dependent.

Case (2). Suppose that $\tau(s(x))=\infty$, $\tau(p)<\infty$, and $q=s(x)$.
Let $x=u\abs{x}$ be the polar decomposition of $x$.  By (\ref{eq:noncom4.1})
\[
 x=\frac1{k+1}\sum_{i=1}^{k+1}uV(f_i),
\]
where $uV(f_i)\in B_{\nonsp}$, $i=1,2,\dots, k+1$.
Since $x$ is $k$-extreme there exist constants $C_1,C_2,\dots, C_{k+1}$, such that $\sum_{i=1}^{k+1}C_i\neq 0$ and $\sum_{i=1}^{k+1}C_i uV(f_i)=0$. However $q=s(x)$ is an identity in the von Neumann algebra $\mathcal{N}\subset s(x)\M s(x)$ and so $u^*uV(f_i)=s(x)V(f_i)=V(f_i)$. Consequently, 
\[
\sum_{i=1}^{k+1}C_i V(f_i)=0
\]
and since $V$ is injective $f_1,f_2,\dots, f_{k+1}$ are linearly dependent.

Case (3). Consider now the case when $q=p=e^{\abs{x}}(\mu(\infty,x),\infty)$ and $\tau(p)=\infty$. Thus in view of Lemma \ref{lm:cond}, $q^{\perp}=e^{\abs{x}}\{0\}+e^{\abs{x}}\{\mu(\infty,x)\}\geq e^{\abs{x}}\{\mu(\infty,x)\}$.

For each $i=1,2,\dots, k+1$, choose $0\leq \alpha_i\leq \mu(\infty,f_i)$ such that $\frac1{k+1}\sum_{i=1}^{k+1}\alpha_i=\mu(\infty,x)$. Such constants exist, since by \cite[Lemma 2.5]{FK} for all $t>0$
\[
\mu(t, x)=\mu\left(t,\frac1{k+1}\sum_{i=1}^{k+1}fi\right)\leq \frac1{k+1}\sum_{i=1}^{k+1}\mu\left(\frac t{k+1}, f_i\right), 
\]
and so $\mu(\infty,x)\leq\frac1{k+1}\sum_{i=1}^{k+1}\mu(\infty, f_i)$.

Define operators $x_i=V(f_i)+\alpha_i e^{\abs{x}}\{\mu(\infty,x)\}$. Observe that since $q$ is an identity in $\mathcal{N}$, $q^{\perp}V(f_i)=V(f_i)q^{\perp}=0$, and so $e^{\abs{x}}\{\mu(\infty,x)\}V(f_i)=
V(f_i)e^{\abs{x}}\{\mu(\infty,x)\}=0$. Furthermore $\alpha_i\leq \mu(\infty,f_i)=\mu(\infty,V(f_i))$, and hence by Corollary \ref{cor:1}, $\mu(x_i)=\mu(V(f_i))=\mu(f_i)$. Hence $x_i\in B_{\nonsp}$ for all $i=1,2,\dots, k+1$. We have now by (\ref{eq:noncom4.1}) that 
\begin{align*}
\abs{x}&=\abs{x}q+\abs{x}e^{\abs{x}}\{\mu(\infty,x)\}
=\abs{x}q+\mu(\infty,x)e^{\abs{x}}\{\mu(\infty,x)\}\\
&=\frac1{k+1}\sum_{i=1}^{k+1}V(f_i)
+\frac1{k+1}\sum_{i=1}^{k+1}\alpha_ie^{\abs{x}}\{\mu(\infty,x)\}=\frac1{k+1}\sum_{i=1}^{k+1}x_i.
\end{align*}
Using the polar decomposition $x=u\abs{x}$, 
\[
x=\frac1{k+1}\sum_{i=1}^{k+1}ux_i=\frac1{k+1}\sum_{i=1}^{k+1}(uV(f_i)+\alpha_iue^{\abs{x}}\{\mu(\infty,x)\}),
\]
and $ux_1,ux_2,\dots, ux_{k+1}$ are linearly dependent. Since two components of $x_i$, $uV(f_i)$ and $\alpha_iue^{\abs{x}}\{\mu(\infty,x)\}$ have disjoint supports, $uV(f_1), \dots, uV(f_{k+1})$ are linearly dependent. Moreover $q\leq s(x)$, and so $u^*uV(f_i)=s(x)V(f_i)=s(x)qV(f_i)=qV(f_i)=V(f_i)$.  Since  $V$ is an isometry, $f_1,f_2,\dots, f_{k+1}$ are linearly dependent. 
\end{proof}

In order to show the converse statement of Theorem \ref{thm:noncom4} we need several preliminary results.
\begin{lemma}
\label{lm:mueq}
Let $x\in S_{\nonsp}$ and $\mu(x)$ be a $k$-extreme point of $B_E$. If $x=\frac{1}{k+1}\sum_{i=1}^{k+1}x_i$, where $x_i\in B_{\nonsp}$, then  $\mu(x)=\frac{1}{k+1}\sum_{i=1}^{k+1}\mu(x_i)$.
\end{lemma}
\begin{proof}
Let $x\in S_{\nonsp}$ and $x=\frac{1}{k+1}\sum_{i=1}^{k+1}x_i$, where $x_i\in B_{\nonsp}$. Since $\mu(x)\prec \frac{1}{k+1}\sum_{i=1}^{k+1}\mu(x_i)$, where $\frac{1}{k+1}\sum_{i=1}^{k+1}\mu(x_i)\in B_E$, the claim follows by Corollary \ref{cor:orbitmain}.
\end{proof}
\begin{proposition}
\label{lm:3}
Let  $x\in S_{E\Mtau}$ be such that  $\mu(x)$ is a $k$-extreme point of $B_{E}$, $\mu(\infty,x)>0$ and $x\geq \mu(\infty,x)s(x)$. Let $x=\frac1{k+1}\sum_{i=1}^{k+1} x_i$, where $x_i\in S_{E\Mtau}$, $x_i\geq 0$,  $i=1,2,\dots,k+1$. Then for every $i=1,2,\dots,k+1$, $x_i\geq \mu(\infty,x_i)s(x_i)$. Moreover, if for some $i$, $\mu(\infty,x_i)>0$ then $s(x_i)=s(x)$.
\end{proposition}
\begin{proof}
Let $\mu(x)$ be $k$-extreme in $E$, $\mu(\infty,x)>0$, $x\geq\mu(\infty,x)s(x)$, and $x=\frac1{k+1}\sum_{i=1}^{k+1} x_i$, where $x_i\in S_{E\Mtau}$, $x_i\geq 0$, for $i=1,2,\dots,k+1$.
By the assumption $\mu(\infty,x)>0$, $\tauone=\infty$.

Observe first that we can assume without loss of generality that $\M$ is non-atomic. If not, we can consider new elements $\mathds{1}\overline\otimes x,\mathds{1}\overline\otimes x_i\in S_{E(\mathcal{A},\kappa)}$, $i=1,2,\dots,k+1$, where $\mathcal{A}=\mathcal{N}\otimes \M$ is a non-atomic von Neumann algebra (see Remark \ref{rm2}) and $\mathds{1}\overline\otimes x=\frac1{k+1}\sum_{i=1}^{k+1}\mathds{1}\overline\otimes x_i$. 
 Observe that  $\tilde{\mu}(\infty,\mathds{1}\overline\otimes x)=\mu(\infty,x)$ and $\tilde{\mu}(\infty,\mathds{1}\overline\otimes x_i)=\mu(\infty,x_i)$. Moreover, $s(x)=s(x_i)$ if and only if  $s(\mathds{1}\overline\otimes x)=\mathds{1}\otimes s(x)=\mathds{1}\otimes s(x_i)=s(\mathds{1}\overline\otimes x_i)$.
 Finally, $\mathds{1}\overline\otimes x-\tilde{\mu}(\infty,\mathds{1}\overline\otimes x)s(\mathds{1}\overline\otimes x)=\mathds{1}\overline\otimes x-\mu(\infty,x)\mathds{1}\otimes s(x)=\mathds{1}\overline\otimes(x-\mu(\infty,x)
  s(x))$, and so  $x\geq \mu(\infty,x)s(x)$ if and only if $\mathds{1}\overline\otimes x\geq \mu(\infty,\mathds{1}\overline\otimes x)s(\mathds{1}\overline\otimes x)$. The same is true for $x_i$.
Therefore all the conditions in the proposition for $x$ and $x_i$ are equivalent to the corresponding conditions for $\mathds{1}\overline\otimes x$ and $\mathds{1}\overline\otimes x_i$.

We will show first that if $\mu(\infty, x_i)>0$ then $s(x_i)=s(x)$, $i=1,2,\dots, k+1$.
Observe that $0=n(x)xn(x)=\frac1{k+1}\sum_{i=1}^{k+1}n(x)x_in(x)$, where $n(x)x_in(x)\geq 0$. Hence $n(x)x_in(x)=0$, and so $x_in(x)=0$. Consequently  $n(x_i)\geq n(x)$, $i=1,2,\dots, k+1$.  Thus by Lemma \ref{lm:mueq},
\begin{equation}
\label{eq:lm3.6}
\begin{split}
x+\mu(\infty,x)n(x)&=\frac{1}{k+1}\sum_{i=1}^{k+1}x_i+\mu(\infty,x)n(x)=\frac{1}{k+1}\sum_{i=1}^{k+1}x_i
\\&+\frac{1}{k+1}\sum_{i=1}^{k+1}\mu(\infty,x_i)n(x)\leq \frac{1}{k+1}\sum_{i=1}^{k+1} (x_i+\mu(\infty,x_i)n(x_i)).
\end{split}
\end{equation}
Note that by Proposition \ref{prop:singfun} (1), $\mu(x+\mu(\infty,x)n(x))=\mu(x)$ and $\mu(x_i+\mu(\infty,x_i)n(x_i))=\mu(x_i)$, $i=1,2,\dots,k+1$. Furthermore, since $\mu(x)$ is $k$-extreme, Lemma \ref{lm:mueq} implies that $\mu(x)=\frac{1}{k+1}\sum_{i=1}^{k+1}\mu(x_i)$. Therefore in view of (\ref{eq:lm3.6}),
\begin{align*}
\mu(x)&=\mu(x+\mu(\infty,x)n(x))\leq \mu\left(\frac1{k+1}\sum_{i=1}^{k+1}
\left(x_i+\mu(\infty,x_i)n(x_i)\right)\right)\\
&\prec \frac1{k+1}\sum_{i=1}^{k+1}\mu(x_i+\mu(\infty,x_i)n(x_i))=\frac1{k+1}\sum_{i=1}^{k+1}\mu(x_i)=\mu(x),
\end{align*}
and so
\[
\mu(x+\mu(\infty,x)n(x))=\mu\left(\frac{1}{k+1}\sum_{i=1}^{k+1} (x_i+\mu(\infty,x_i)n(x_i))\right).
\]

 However, $x+\mu(\infty,x)n(x)\leq \frac{1}{k+1}\sum_{i=1}^{k+1} (x_i+\mu(\infty,x_i)n(x_i))$ and $x+\mu(\infty,x)n(x)\geq \mu(\infty,x)\one$. By Proposition \ref{prop:singfun} (5) and in view of (\ref{eq:lm3.6}),
\begin{align}
\label{eq:lm3.7}
x+\mu(\infty,x)n(x)&=\frac{1}{k+1}\sum_{i=1}^{k+1}\left(x_i+\mu(\infty,x_i)n(x_i)\right)
\\ \notag
&=\frac{1}{k+1}\sum_{i=1}^{k+1}\left(x_i+\mu(\infty,x_i)n(x)\right).
\end{align} Hence if $\mu(\infty,x_i)>0$ then $n(x)=n(x_i)$ and $s(x)=s(x_i)$.

We will show next that $x_i\geq \mu(\infty,x_i)s(x_i)$, $i=1,2,\dots,k+1$.

Assume first that $s(x)=\one$. Hence $n(x)=e^x\{0\}=0$. Then  $x\geq \mu(\infty,x)\one$ and in view of Lemma \ref{lm:cond}, $n(x)=e^x[0,\mu(\infty,x))=0$. Moreover, as observed in Lemma \ref{lm:mueq}, $\mu(x)=\frac1{k+1}\sum_{i=1}^{k+1}\mu(x_i)$.

Let $0<\lambda<\tauone=\infty$. Choose by  Lemma \ref{lm:aux1} a projection $p_\lambda$ such that $\tau(p_{\lambda})=\lambda$ and
\begin{equation}
\label{eq:lm3.1}
e^x\left(\mu(\lambda,x),\infty)\right)\leq p_{\lambda}\leq e^x\left[\mu(\lambda,x),\infty)\right).
\end{equation}
By Lemma \ref{lm:aux2}, $\tau(xp_{\lambda})=\int_0^{\tau(p_{\lambda})}
\mu(x)$. Therefore
\begin{align*}
\frac{1}{k+1}\sum_{i=1}^{k+1}\tau(x_ip_{\lambda})&=\tau(xp_{\lambda})=\int_0^{\tau(p_{\lambda})}\mu(x)
=\frac{1}{k+1}\sum_{i=1}^{k+1}\int_0^{\tau(p_{\lambda})}\mu(x_i).
\end{align*}
Since $\mu(x_ip_{\lambda})\prec\mu(x_i)\chi_{[0,\tau(p_{\lambda}))}$, and so $\tau(x_ip_{\lambda})=\int_0^{\infty}
\mu(x_ip_{\lambda})\leq \int_0^{\tau(p_{\lambda})} \mu(x_i)$, it follows that
\[
\tau(x_ip_{\lambda})= \int_0^{\tau(p_{\lambda})} \mu(x_i),
\]
for $i=1,2,\dots,k+1$.
Consequently, Lemma \ref{lm:aux3} implies that 
\begin{equation}
\label{eq:lm3.2}
e^{x_i}\left(\mu(\lambda,x_i),\infty)\right)\leq p_{\lambda}\leq e^{x_i}\left[\mu(\lambda,x_i),\infty)\right),\quad i=1,2,\dots,k+1.
\end{equation}
Denote by $p=\vee_{\lambda>0} p_{\lambda}$. Since $\vee_{\lambda>0} e^x\left(\mu(\lambda,x),\infty\right)=e^x\left(\mu(\infty,x),\infty\right)$ and $\vee_{\lambda>0}  e^x\left[\mu(\lambda,x),\infty\right)\leq e^x\left[\mu(\infty,x),\infty\right)$, relation  (\ref{eq:lm3.1}) implies that
\[
e^x\left(\mu(\infty,x),\infty\right)\leq p\leq e^x\left[\mu(\infty,x),\infty\right).
\]
Similarly by (\ref{eq:lm3.2}),
\[
e^{x_i}\left(\mu(\infty,x_i),\infty\right)\leq p\leq e^{x_i}\left[\mu(\infty,x_i),\infty\right),\quad i=1,2,\dots,k+1.
\]
Let $j\in\{1,2,\dots,k+1\}$ be fixed.
Then
\begin{align}
\label{eq:lm3.3}
e^x\left(\mu(\infty,x),\infty\right)&\leq e^{x_j}\left[\mu(\infty,x_j),\infty\right),
\end{align}
and
\begin{align}
\label{eq:lm3.4}
e^{x_i}\left(\mu(\infty,x_i),\infty\right)&\leq e^{x_j}\left[\mu(\infty,x_j),\infty\right)\text{ for all }i=1,2,\dots,k+1.
\end{align}
 
Assume that $\mu(\infty, x_j)>0$. We will show next that  $e^{x_j}[0,\mu(\infty,x_j))=0$.  Suppose to the contrary that $e^{x_j}[0,\mu(\infty,x_j))\neq 0$. Choose $0<\lambda_0<\mu(\infty,x_j)$ such that $e^{x_j}[0,\lambda_0]\neq 0$. Since $\M$ is non-atomic and $\tau$ is semi-finite, there exists a non-zero projection $q\leq e^{x_j}[0,\lambda_0]$ and $q\neq e^{x_j}[0,\lambda_0]$, with $\tau(q)<\infty$.  Note first that by (\ref{eq:lm3.3}) and the fact that $e^{x}[0,\mu(\infty,x))=0$, $q\leq e^{x_j}[0,\mu(\infty,x_j))\leq e^x[0,\mu(\infty,x)]=e^x\{\mu(\infty,x)\}$.   Hence by spectral representation of $x$
\begin{align*}
xq&=xe^x\{\mu(\infty,x)\}q=\mu(\infty,x)e^x\{\mu(\infty,x)\}q=\mu(\infty,x)q.
\end{align*}
Consequently, 
\begin{align}
\label{eq:lm3.5}
\sum_{i=1}^{k+1}\tau(x_iq)&=\tau\left(\sum_{i=1}^{k+1}x_iq
\right)=\tau((k+1)xq)=(k+1)\tau(\mu(\infty,x)q)\\\notag
&=\sum_{i=1}^{k+1}\mu(\infty,x_i)\tau(q).
\end{align}
By (\ref{eq:lm3.4}), $q\leq e^{x_j}[0,\mu(\infty,x_j))\leq e^{x_i}[0,\mu(\infty,x_i)]$ for all $i=1,2,\dots,k+1$. Since $0\leq x_ie^{x_i}[0,\mu(\infty,x_i)]\leq\mu(\infty,x_i)e^{x_i}[0,\mu(\infty,x_i)]$ we have
\begin{align*}
\tau(x_iq)&=\tau(x_ie^{x_i}[0,\mu(\infty,x_i)]q)=\tau(qx_ie^{x_i}[0,\mu(\infty,x_i)]q) \\
&\leq \tau(q\mu(\infty,x_i)e^{x_i}[0,\mu(\infty,x_i)]q)=\tau(\mu(\infty,x_i)e^{x_i}[0,\mu(\infty,x_i)]q)\\
&=\tau(\mu(\infty,x_i)q)=\mu(\infty,x_i)\tau(q),
\end{align*}
for all $i=1,2,\dots,k+1$. Therefore (\ref{eq:lm3.5}) implies that 
\[
\tau(x_iq)=\mu(\infty,x_i)\tau(q),\,\text{ for all }i=1,2,\dots,k+1.
\]
 In particular we get that $\tau(x_jq)=\mu(\infty,x_j)\tau(q)$. In view of $0<\tau(q)<\infty$ we must have
\begin{align*}
\tau(x_jq)&=\tau(qx_je^{x_j}[0,\lambda_0]q)\leq \tau(q\lambda_0e^{x_j}[0,\lambda_0]q)= \tau(\lambda_0q)=\lambda_0\tau(q)\\
&<\mu(\infty,x_j)\tau(q),
\end{align*}
 which is impossible.
Consequently $e^{x_j}[0,\mu(\infty,x_j))=0$. By the first part of the proof, $n(x)=n(x_j)=e^{x_j}\{0\}=0$, and by Lemma \ref{lm:cond}, $x_j\geq\mu(\infty,x_j)\one$.  Since the same follows instantly for those $x_i$'s for which $\mu(\infty,x_i)=0$,  we have that $x_i\geq \mu(\infty,x_i)\one$ for all $i=1,2,\dots, k+1$.

Consider now the general case, when $s(x)$ is not necessarily an identity. Recall that by Proposition \ref{prop:singfun} (1), $\mu(x+\mu(\infty,x)n(x))=\mu(x)$ and $\mu(x_i+\mu(\infty,x_i)n(x_i)=\mu(x_i)$, $i=1,2,\dots,k+1$. We also have by (\ref{eq:lm3.7}) that $ x+\mu(\infty,x)n(x)=\frac{1}{k+1}\sum_{i=1}^{k+1}\left(x_i+\mu(\infty,x_i)
n(x_i)\right)$. Clearly $s(x+\mu(\infty,x)n(x))=\one$ and $x+\mu(\infty,x)n(x)\geq \mu(\infty,x)\one$.

We will repeat the above argument to the operators $x+\mu(\infty,x)n(x)$, $x_i+\mu(\infty,x_i)n(x_i)$ instead of $x$, $x_i$, respectively. Consequently, it follows that $x_i+\mu(\infty,x_i)n(x_i)\geq \mu(\infty,x_i)\one$ and  $x_i\geq \mu(\infty,x_i)s(x_i)$  for all  $i=1,2,\dots,k+1$.
\end{proof}

\begin{lemma}
\label{lm:4}
Suppose $\mu(x)$ is $k$-extreme, $x=\frac1{k+1}\sum_{i=1}^{k+1}b_i\leq \frac1{k+1}\sum_{i=1}^{k+1}a_i$, $a_i,b_i\in B_{\nonsp}$ and $a_i\prec b_i$, $i=1,2,\dots, k+1$. Then $\mu(a_i)=\mu(b_i)$ for all $i=1,2,\dots, k+1$.

If in addition $\mu(\infty,x)=0$ then $a_i=b_i$, for all $i=1,2,\dots, k+1$. 
\end{lemma}
\begin{proof}
By Lemma \ref{lm:mueq} and in view of $a_i\prec b_i$, $\mu(x)=\frac1{k+1}\sum_{i=1}^{k+1}\mu(b_i)\leq \mu(\frac1{k+1}\sum_{i=1}^{k+1}a_i)\prec  \frac1{k+1}\sum_{i=1}^{k+1}\mu(a_i)\prec\frac1{k+1}\sum_{i=1}^{k+1}\mu(b_i)=\mu(x)$.
Hence  $\mu(x)=\frac1{k+1}\sum_{i=1}^{k+1}\mu(a_i)=\frac1{k+1}\sum_{i=1}^{k+1}\mu(b_i)$, and so  for all $t>0$,
\[
\sum_{i=1}^{k+1}\int_0^t \left(\mu(b_i)-\mu(a_i)\right)=0.
\]
Since for all $i=1,2,\dots, k+1$, $\mu(a_i)\prec \mu(b_i)$ we have that $\int_0^t \left(\mu(b_i)-\mu(a_i)\right)\geq 0$, $t>0$. Therefore $\mu(a_i)=\mu(b_i)$, $i=1,2,\dots, k+1$.

Suppose now that $\mu(\infty,x)=0$. Then  clearly $\mu(\infty,a_i)=\mu(\infty,b_i)=0$ for all $i=1,2,\dots, k+1$. Note that
\[
x=\frac1{k+1}\sum_{i=1}^{k+1}b_i\leq \frac1{k+1}\sum_{i=1}^{k+1}\frac{a_i+b_i}2,
\]
where $\frac{a_i+b_i}2\prec b_i$.
By the previous argument, $\mu(b_i)=\mu\left(\frac{a_i+b_i}2\right)$. Therefore $\mu(a_i)=\mu(b_i)=\mu\left(\frac{a_i+b_i}2\right)$, and Proposition \ref{prop:singfun} (3) implies that $a_i=b_i$, $i=1,2,\dots, k+1$.
\end{proof}
\begin{lemma}
\label{lm:8}
Suppose that $\mu(x)$ is $k$-extreme, $\abs{x}\geq \mu(\infty,x)s(x)$, and $x=\frac1{k+1}\sum_{i=1}^{k+1}a_i= \frac1{k+1}\sum_{i=1}^{k+1}b_i$, $a_i,b_i\in B_{\nonsp}$. If $a_i\prec b_i$ and $a_i,b_i\geq 0$ for all $i=1,2,\dots, k+1$ then $a_i=b_i$, $i=1,2,\dots, k+1$. 
\end{lemma}
\begin{proof}
Note that if $a_i\geq 0 $ for all $i=1,2,\dots, k+1$, then $x=\abs{x}$.
Since $\abs{x}=\frac1{k+1}\sum_{i=1}^{k+1}a_i= \frac1{k+1}\sum_{i=1}^{k+1}b_i=\frac1{k+1}\sum_{i=1}^{k+1}\frac{a_i+b_i}{2}$, by Lemma \ref{lm:4},
\[
\mu(a_i)=\mu(b_i)=\mu\left(\frac{a_i+b_i}{2}\right).
\]
 Denote by $C_i=\mu(\infty, a_i)=\mu(\infty, b_i)=\mu(\infty, \frac{a_i+b_i}2)$. 
 In case of $\mu(\infty, x)>0$,  Proposition \ref{lm:3} guarantees that  if $C_i>0$ then $a_i\geq C_is(x)$, $b_i\geq C_is(x)$ and $\frac{a_i+b_i}2\geq C_is(x)$, and also $s(a_i)=s(b_i)=s(\frac{a_i+b_i}2)=s(x)$. Hence in view of Proposition \ref{prop:singfun} (2),
\begin{align*}
\mu(a_i-C_is(x))&=\mu(a_i)-C_i=\mu(b_i)-C_i=\mu(b_i-C_is(x))=\mu\left(\frac{a_i+b_i}2\right)\\
-C_i&=\mu\left(\frac{a_i+b_i}2-C_is(x)\right)=\mu\left(\frac{a_i-C_is(x)+b_i-C_is(x)}2\right).
\end{align*}
Clearly if $C_i=0$ the above equalities hold.  We have now that $a_i-C_is(x), b_i-C_is(x)\in S_0\Mtau$, and so by Proposition \ref{prop:singfun} (3) it follows that $a_i=b_i$, $i=1,2, \dots, k+1$.
\end{proof}
\begin{lemma}
\label{lm:5}
Suppose that $\mu(x)$ is $k$-extreme, $\abs{x}=\frac1{k+1}\sum_{i=1}^{k+1}x_i$, $x_i\in B_{\nonsp}$ and either $\mu(\infty,x)=0$ or $\abs{x}\geq \mu(\infty,x)s(x)$. Then $n(x)x_i=x_in(x)=0$ if and only if $x_i\geq 0$, $i=1,2,\dots, k+1$.
\end{lemma}
\begin{proof}
Assume first that for all $j=1,2,\dots, k+1$, $n(x)x_j=x_jn(x)=0$. Denote by $\Rep(x_j)=\dfrac{x_j+x_j^*}{2}$, $j=1,2,\dots, k+1$.
Since $\Rep(x_j)\leq \abs{\Rep(x_j)}$ we have that
\begin{equation}
\label{eq:lm5.1}
\abs{x}=\frac1{k+1}\sum_{j=1}^{k+1}x_j
=\frac1{k+1}\sum_{j=1}^{k+1}\Rep(x_j)\leq \frac1{k+1}\sum_{j=1}^{k+1}\abs{\Rep(x_j)}.
\end{equation}
By Lemma \ref{lm:mueq}, $\mu(x)\leq \mu(\frac1{k+1}\sum_{j=1}^{k+1}\abs{\Rep(x_j)})
\prec\frac1{k+1}\sum_{j=1}^{k+1}\mu(\Rep(x_j))=\mu(x)$, and so 
\[
\mu(x)=\mu\left(\frac1{k+1}\sum_{j=1}^{k+1}
\abs{\Rep(x_j)}\right).
\]
By assumption $x_jn(x)=n(x)x_j=0$, and so $\Rep(x_j)n(x)=n(x)\Rep(x_j)=0$. Thus
$n(x)\left(\frac1{k+1}\sum_{j=1}^{k+1}\Rep(x_j)\right)
=\left(\frac1{k+1}\sum_{j=1}^{k+1}\Rep(x_j)\right)n(x)=0$.
Denote by $C=\mu(\infty,x)=\mu\left(\infty,\frac1{k+1}\sum_{j=1}^{k+1}\abs{\Rep(x_j)}\right)$. By Corollary \ref{cor:1}, 
\begin{align*}
\mu(\abs{x}+Cn(x))&=\mu(x)=\mu\left(\frac1{k+1}
\sum_{j=1}^{k+1}\abs{\Rep(x_j)}\right)\\
&=\mu\left(\frac1{k+1}\sum_{j=1}^{k+1}\abs{\Rep(x_j)}
+Cn(x)\right).
\end{align*}
Since $C\one\leq\abs{x}+Cn(x)\leq \frac1{k+1}\sum_{j=1}^{k+1}\abs{\Rep(x_j)}+Cn(x)$, Proposition \ref{prop:singfun} (5) implies that $\abs{x}=\frac1{k+1}\sum_{j=1}^{k+1}\abs{\Rep(x_j)}
=\frac1{k+1}\sum_{j=1}^{k+1}\Rep(x_j)$.
Since for all $j=1,2,\dots, k+1$, $\Rep(x_j)\leq \abs{\Rep(x_j)}$ we get the equality $\Rep(x_j)=\abs{\Rep(x_j)}$.  

We will show next that $\Imp(x_j)=0$ and therefore $x_j=\Rep(x_j)$.
Note that $\Rep(x_j)\prec x_j$, $j=1,2,\dots,k+1$. Thus by Lemma \ref{lm:4} and (\ref{eq:lm5.1}) we have $\mu(x_j)=\mu(\Rep(x_j))$. Let $C_j=\mu(\infty,x_j)=\mu(\infty,\Rep(x_j))$, $j=1,2,\dots, k+1$. Then 
\[
\mu(x_j+C_jn(x))=\mu(x_j)=\mu(\Rep(x_j))=\mu(\Rep(x_j)
+C_jn(x)), 
\]
and therefore
\[
\mu(\Rep(x_j)+C_jn(x)+i\Imp(x_j))=\mu(\Rep(x_j)+C_jn(x)).
\]
Observe that by (\ref{eq:lm5.1}), if $\mu(\infty, x)>0$ and $C_j>0$, Proposition \ref{lm:3} implies that $\Rep(x_j)+C_jn(x)\geq C_j\one$. Clearly, the same is true for $C_j=0$ and also for the case $\mu(\infty, x)=0$ since $\mu(x)=\frac1{k+1}\sum_{j=1}^{k+1}\mu(x_j)$. Thus by Proposition \ref{prop:singfun} (6) for all $j=1,2,\dots, k+1$, $\Imp(x_j)=0$ and $x_j=\Rep(x_j)\geq 0$.

For the converse assume that $x_j\geq 0$, for all $j=1,2,\dots, k+1$. Then $0=n(x)\abs{x}n(x)=\frac1{k+1}\sum_{j=1}^{k+1}n(x)x_jn(x)$.
Since $n(x)x_jn(x)\geq 0$ it follows that $n(x)x_jn(x)=0$.
Consequently $x_jn(x)=0$ and $n(x)x_j=(x_jn(x))^*=0$, for all $j=1,2,\dots, k+1$ .
\end{proof}
\begin{lemma}
\label{lm:6}
Suppose that $\mu(x)$ is $k$-extreme, $\abs{x}=\frac1{k+1}\sum_{i=1}^{k+1}x_i$, $x_i\in B_{\nonsp}$, $i=1,2,\dots, k+1$, and  $\mu(\infty,x)=0$. Then $x_i\geq 0$ for all $i=1,2,\dots, k+1$.
\end{lemma}
\begin{proof}
Consider the equations
\[
\abs{x}=\frac1{k+1}\sum_{i=1}^{k+1}x_i
=\frac1{k+1}\sum_{i=1}^{k+1}x_i^*\quad i=1,2,\dots,k+1.
\]
By Lemma \ref{lm:4}, $x_i=x_i^*$, and consequently
\[
\abs{x}=\frac1{k+1}\sum_{i=1}^{k+1}x_i\leq \frac1{k+1}\sum_{i=1}^{k+1}\abs{x_i}.
\]
By Lemma \ref{lm:mueq}, $\mu(x)=\frac1{k+1}\sum_{i=1}^{k+1}\mu(x_i)$, and so $\mu(x)=\mu( \frac1{k+1}\sum_{i=1}^{k+1}\abs{x_i})$.
As an immediate consequence of Proposition \ref{prop:singfun} (5) we have that
$\abs{x}=\frac1{k+1}\sum_{i=1}^{k+1}\abs{x_i}=
\frac1{k+1}\sum_{i=1}^{k+1}x_i$.  Since $x_i\leq \abs{x_i}$ the equality $x_i=\abs{x_i}$ follows. 
\end{proof}
\begin{lemma}
\label{lm:7}
Let $\mu(x)$ be $k$-extreme and $\abs{x}=\frac1{k+1}\sum_{i=1}^{k+1}\abs{x_i}$, $x_i\in B_{\nonsp}$, $i=1,2,\dots, k+1$.  If $\abs{x}\geq \mu(\infty,x)\one$ then for all $i=1,2,\dots, k+1$, \[
\abs{x_i}e^{\abs{x}}\{\mu(\infty,x)\}=\mu(\infty, x_i)e^{\abs{x}}\{\mu(\infty,x)\}.
\]
\end{lemma}
\begin{proof}

Observe first that in view of Lemma \ref{lm:mueq}, if $\mu(\infty,x)=0$  then we must have $\mu(\infty,x_i)=0$ for all $i=1,2,\dots, k+1$. Then the hypothesis becomes $\abs{x_i}n(x)=0$, and it follows by Lemma \ref{lm:5}.

Assume now that $\mu(\infty,x)>0$ and therefore $\one=s(x)=e^{\abs{x}}[\mu(\infty,x),\infty)$ (see Lemma \ref{lm:cond}). We will show first that for all $i=1,2,\dots, k+1$ we have
\[
e^{\abs{x}}\{\mu(\infty,x)\}\wedge e^{\abs{x_i}}(\mu(\infty,x_i),\infty)=0.
\]
Fix $j\in\{1,2,\dots, k+1\}$ and assume that for some  $s>\mu(\infty, x_j)$ the projection
\[
p_j=e^{\abs{x}}\{\mu(\infty,x)\}\wedge e^{\abs{x_j}}(s,\infty)\neq 0.
\]
Clearly $\tau(p_j)\leq \tau(e^{\abs{x_j}}(s,\infty))<\infty$ and 
\[
\abs{x}p_j=\abs{x}e^{\abs{x}}\{\mu(\infty,x)\}p_j=\mu(\infty,x)e^{\abs{x}}\{\mu(\infty,x)\}p_j=\mu(\infty,x)p_j.
\]
Moreover, by Proposition \ref{lm:3} we have that for all $i=1,2,\dots, k+1$, $\abs{x_i}\geq \mu(\infty,x_i)s(x_i)$.  We also have that if $\mu(\infty, x_i)>0$ then $s(x_i)=s(x)=e^{\abs{x}}[\mu(\infty,x),\infty)$, and so $p_j\leq s(x_i)$, $i=1,2,\dots,k+1$. Hence it follows that
\[
p_j\abs{x_i}p_j\geq \mu(\infty,x_i)p_js(x_i)p_j=\mu(\infty,x_i)p_j,
\]
and $\tau(\abs{x_i}p_j)=\tau(p_j\abs{x_i}p_j)\geq \mu(\infty,x_i)\tau(p_j)$. Furthermore, 
\[
p_j\abs{x_j}p_j=p_j\abs{x_j}e^{\abs{x_j}}(s,\infty)p_j\geq sp_je^{\abs{x_j}}(s,\infty)p_j=sp_j,
\]
and consequently $\tau(\abs{x_j}p_j)\geq s\tau(p_j)>\mu(\infty,x_j)\tau(p_j)$.

Recall that if $\mu(x)$ is $k$-extreme and $\abs{x}=\frac1{k+1}\sum_{i=1}^{k+1}\abs{x_i}$ then by Lemma \ref{lm:mueq}, $\mu(x)=\frac1{k+1}\sum_{i=1}^{k+1}\mu(x_i)$. Therefore
\begin{align*}
\mu(\infty,x)\tau(p_j)&=\tau(\abs{x}p_j)
=\frac1{k+1}\sum_{i=1}^{k+1}\tau(\abs{x_i}p_j)
=\frac1{k+1}\sum_{\substack{i=1,\,i\neq j}}^{k+1}\tau(\abs{x_i}p_j)\\
+\frac{1}{k+1}\tau(\abs{x_j}p_j)&>\frac1{k+1}\sum_{\substack{i=1,\,i\neq j}}^{k+1}\mu(\infty,x_i)\tau(p_j)+\frac1{k+1}\mu(\infty,x_j)\tau(p_j)\\
&=\frac1{k+1}\sum_{i=1}^{k+1}\mu(\infty,x_i)\tau(p_j)=\mu(\infty,x)\tau(p_j),
\end{align*}
which contradicts the assumption that $p_j\neq 0$. Consequently $e^{\abs{x}}\{\mu(\infty,x)\}\wedge e^{\abs{x_j}}(s,\infty)=0$ for all $s>\mu(\infty,x_j)$ and
\[
e^{\abs{x}}\{\mu(\infty,x)\}\wedge e^{\abs{x_j}}(\mu(\infty,x_j),\infty)=0.
\]
If $\mu(\infty,x_j)=0$, then clearly $e^{\abs{x_j}}(\mu(\infty,x_j),\infty)^{\perp}=e^{\abs{x_j}}\{\mu(\infty, x_j)\}$. On the other hand, if $\mu(\infty,x_j)>0$, then by Proposition \ref{lm:3} and Lemma \ref{lm:cond}, $e^{\abs{x_j}}\{0\}=n(x_j)=n(x)=0$ and $e^{\abs{x_j}}(0,\mu(\infty,x_j))=0$. Hence we also have  $e^{\abs{x_j}}(\mu(\infty,x_j),\infty)^{\perp}=e^{\abs{x_j}}\{\mu(\infty,x_j)\}$. Thus $e^{\abs{x}}\{\mu(\infty,x)\}=e^{\abs{x}}\{\mu(\infty,x)\}\wedge (e^{\abs{x_j}}(\mu(\infty,x_j),\infty)\vee e^{\abs{x_j}}\{\mu(\infty,x_j)\})=e^{\abs{x}}\{\mu(\infty,x)\}\wedge e^{\abs{x_j}}\{\mu(\infty,x_j)\}\leq e^{\abs{x_j}}\{\mu(\infty,x_j)\}$. Therefore
\begin{align*}
x_je^{\abs{x}}\{\mu(\infty,x)\}&
=x_je^{\abs{x_j}}\{\mu(\infty,x_j)\}e^{\abs{x}}
\{\mu(\infty,x)\}\\
&=\mu(\infty,x_j)e^{\abs{x_j}}\{\mu(\infty,x_j)\}e^{\abs{x}}\{\mu(\infty,x)\}\\
&=\mu(\infty,x_j)e^{\abs{x}}\{\mu(\infty,x)\},
\end{align*}
and since $j$ was arbitrary the claim follows.
\end{proof}
\begin{theorem}
\label{thm:noncom5}
Let $\M$ be  a von Neumann algebra with a faithful, normal, $\sigma$- finite trace $\tau$ and $E$ be a strongly symmetric function space. An element  $x\in S_{\nonsp}$ is a $k$-extreme point of $B_{\nonsp}$ whenever $\mu(x)$ is a $k$-extreme point of $B_E$ and one of the following conditions holds:
\begin{itemize}
\item[(i)] $\mu(\infty,x)=0$,
\item[(ii)] $n(x)\mathcal{M}n(x^*)=0$ and $\abs{x}\geq\mu(\infty,x)s(x)$.
\end{itemize}
\end{theorem}
\begin{proof}
Assume first that $\M$ is non-atomic. 
Suppose that
\begin{equation}
\label{eq:noncom5.1}
x=\frac{1}{k+1}\sum_{i=1}^{k+1} x_i, \text{ where }x_i\in B_{\nonsp},\quad i=1,2,\dots,k+1.
\end{equation}
Let $x=u\abs{x}$ be the polar decomposition of $x$.
Then
\begin{equation}
\label{eq:noncom5.2}
\abs{x}=\frac1{k+1}\sum_{i=1}^{k+1}u^*x_i
=\frac1{k+1}\sum_{i=1}^{k+1}u^*x_is(x),
\end{equation}
and
\begin{equation}
\label{eq:noncom5.4}
x=\frac{1}{k+1}\sum_{i=1}^{k+1} x_i=\frac{1}{k+1}\sum_{i=1}^{k+1} s(x^*)x_i.
\end{equation}

Consider first the case when $\abs{x}\geq \mu(\infty,x)\one$. Note that if $\mu(\infty, x)>0$ then  $s(x)=\one$ and $uu^*=u^*u=\one$.

If $\mu(\infty,x)=0$ it follows by Lemma \ref{lm:6} and (\ref{eq:noncom5.2}) that $u^*x_i\geq 0$ for all $i=1,2,\dots, k+1$.  Moreover, (\ref{eq:noncom5.4}) combined with Lemma \ref{lm:4} implies that $s(x^*)x_i=x_i$.

If $\mu(\infty,x)>0$ and consequently $n(x)=0$, by Lemma \ref{lm:5} we also have that $u^*x_i\geq 0$. Clearly since $s(x^*)=\one$ we also have $x_i=s(x^*)x_i$ for all $i=1,2,\dots, k+1$. 

Furthermore in either case, $\mu(x_i)=\mu(uu^* x_i)\leq \mu(u^*x_i)\leq \mu(x_i)$, and so $\mu(x_i)=\mu(u^*x_i)$, $i=1,2,\dots, k+1$.
Therefore we always have 
\begin{equation}
\label{eq:noncom5.5}
 u^*x_i\geq 0,\quad s(x^*)x_i=x_i,\quad \text{and}\quad \mu(u^*x_i)=\mu(x_i),\quad i=1,2,\dots, k+1.
\end{equation}

Note that by Lemma \ref{lm:mueq}, $\mu(x)=\frac1{k+1}\sum_{i=1}^{k+1}\mu(x_i)$. Since $\mu(x)$ is $k$-extreme, there exist constants $C_1,C_2,\dots, C_{k+1}$ not all vanishing such that 
\begin{equation}
\label{eq:noncom5.6}
\sum_{i=1}^{k+1}C_i\mu(x_i)=0.
\end{equation}

Consider now the operator $\abs{x}-\mu(\infty,x)\one \in S_0^+\Mtau$ and denote by 
\[
p=s(\abs{x}-\mu(\infty,x)\one)=e^{\abs{x}}(\mu(\infty,x),\infty).
\]
Lets define a projection $q$ in the following way.
\begin{itemize}
\item[{(1)}] $q=\one$ if $\tau(s(x))<\infty$, 
\item[{(2)}] $q=s(x)$ if $\tau(s(x))=\infty$ and $\tau(p)<\infty$,
\item[{(3)}] $q=p$ if $\tau(s(x))=\infty$ and $\tau(p)=\infty$.
\end{itemize}  
 Then by Corollary \ref{cor:isom} there are a non-atomic, commutative von Neumann algebra $\mathcal{N}\subset q\M q$ and a $*$-isomorphism $V$ from $S([0,\tauone),m)$ into $S(\mathcal{N},\tau)$ such that 
\[
V\mu(x)=\abs{x}q\quad\text{and}\quad \mu(V(f))=\mu(f),\,\text{for all } f\in S([0,\tauone),m).
\]
By (\ref{eq:noncom5.2}), 
\[
\abs{x}q=\frac{1}{k+1}\sum_{i=1}^{k+1} qu^*x_iq,
\]
 and since $\mu(\abs{x}q)=\mu(x)$ is $k$-extreme, Lemma \ref{lm:mueq} guarantees that 
\[
\frac{1}{k+1}\sum_{i=1}^{k+1} \mu(x_i)=\mu(x)=\mu(\abs{x}q)=\frac{1}{k+1}\sum_{i=1}^{k+1} \mu(qu^*x_iq).
\]
Clearly $\mu(qu^*x_iq)\leq \mu(x_i)$ and so $\mu(qu^*x_iq)=\mu(x_i)$, for all $i=1,2,\dots, k+1$. Moreover applying $V$ to the above equation we get the following
\[
\frac{1}{k+1}\sum_{i=1}^{k+1} qu^*x_iq=\abs{x}q=V\mu(x)=\frac{1}{k+1}\sum_{i=1}^{k+1}V \mu(qu^*x_iq).
\]
We have  $\mu(V\mu(qu^*x_iq))=\mu(qu^*x_iq)=\mu(x_i)$, and $s(\abs{x}q)=q$. 
By Lemma \ref{lm:8}, $V\mu(qu^*x_iq)=qu^*x_iq$, $i=1,2,\dots, k+1$. Applying now $V$ to (\ref{eq:noncom5.6}),
\[
q\left(\sum_{i=1}^{k+1}C_iu^*x_i \right)q=\sum_{i=1}^{k+1}C_iqu^*x_iq
=\sum_{i=1}^{k+1}C_iV\mu(x_i)=0. \]
Hence $\displaystyle \sum_{i=1}^{k+1}C_iu^*x_i q=\left(\sum_{i=1}^{k+1}C_iu^*x_i \right)q=0$ and in view of (\ref{eq:noncom5.5}), $s(x^*)x_i=x_i$, and consequently
\begin{equation}
\label{eq:noncom5.3}
\sum_{i=1}^{k+1}C_ix_iq=0.
\end{equation}

Case 1. Let $\tau(s(x))<\infty$ and $q=\one$. 
Clearly by (\ref{eq:noncom5.3}),
\[
\sum_{i=1}^{k+1}C_ix_i=0.
\]

Case 2. Assume that $\tau(s(x))=\infty$, $\tau(p)<\infty$ and $q=s(x)$. Note that if $\mu(\infty,x)=0$ then $p=s(x)$. Therefore the case is only possible when $\mu(\infty, x)>0$ and then $s(x)=\one$. Therefore $q=\one$ and again by (\ref{eq:noncom5.3}),
\[
\sum_{i=1}^{k+1}C_ix_i=0.
\]

Case 3.  Assume that $\tau(s(x))=\infty$, $\tau(p)=\infty$ and  $q=p$. If $\mu(\infty, x)=0$,  $q=e^{\abs{x}}(0,\infty)$ and $q^{\perp}=e^{\abs{x}}\{0\}= e^{\abs{x}}\{\mu(\infty,x)\}$. If $\mu(\infty,x)>0$ and so $n(x)=e^{\abs{x}}\{0\}=0$, we also have that $q^{\perp}=e^{\abs{x}}\{\mu(\infty,x)\}$.

 Therefore  by Lemma \ref{lm:7} in view of (\ref{eq:noncom5.2}) and (\ref{eq:noncom5.5}) we get $u^*x_iq^{\perp}=\mu(\infty,x_i)q^{\perp}$. By (\ref{eq:noncom5.6}),
\[
\sum_{i=1}^{k+1}C_iu^*x_iq^{\perp}=\left(\sum_{i=1}^{k+1}C_i\mu(\infty,x_i)\right)q^{\perp}=0.
\]
Since by (\ref{eq:noncom5.5}), $s(x^*)x_i=uu^*x_i=x_i$, the above equation becomes
\[
\sum_{i=1}^{k+1}C_ix_iq^{\perp}=0.
\]
This combined with (\ref{eq:noncom5.3}) implies that 
\[
\sum_{i=1}^{k+1}C_ix_i=0.
\]
Consequently in either case $x_1,x_2,\dots, x_{k+1}$ are linearly dependent, and $x$ is $k$-extreme.

We have shown that if $\mu(x)$ is a $k$-extreme point of $B_E$ and $\abs{x}\geq \mu(\infty,x)\one$, then $x$ is a $k$-extreme point of $B_{\nonsp}$. In particular if $\mu(\infty,x)=0$, the claim follows.

Assume now that $\mu(\infty,x)>0$ and (ii) holds. 
By Lemma \ref{lm:mueq} and equations (\ref{eq:noncom5.1}) and (\ref{eq:noncom5.2}), 
\[
\mu(x)=\frac{1}{k+1}\sum_{i=1}^{k+1} \mu(x_i)=\frac{1}{k+1}\sum_{i=1}^{k+1} \mu(u^*x_is(x)).
\]
 Since $\mu(u^*x_is(x))\leq \mu(x_i)$, the equality $\mu(u^*x_is(x))=\mu(x_i)$ holds for all $i=1,2,\dots,k+1$. Moreover by (\ref{eq:noncom5.1}),
\begin{align*}
\abs{x}+\mu(\infty,x)n(x)&=\frac{1}{k+1}\sum_{i=1}^{k+1} u^*x_is(x)+\mu(\infty,x)n(x)\\
&=\frac{1}{k+1}\sum_{i=1}^{k+1} \left(u^*x_is(x)+\mu(\infty,x_i)n(x)\right).
\end{align*}
Clearly $n(x)u^*x_is(x)=u^*x_is(x)n(x)=0$, and $\mu(\infty,x_i)=\mu(\infty,u^*x_is(x))$. Hence by Corollary \ref{cor:1}, $\mu\left(u^*x_is(x)+\mu(\infty,x_i)n(x)\right)=\mu(u^*x_is(x))=\mu(x_i)$, and so $u^*x_is(x)+\mu(\infty,x_i)n(x)\in B_{\nonsp}$, $i=1,2,\dots,k+1$. Since $\mu(\abs{x}+\mu(\infty,x)n(x))=\mu(x)$ is $k$-extreme, where $\abs{x}+\mu(\infty,x)n(x)\geq \mu(\infty,x)\one$, the previous case implies that 
$\abs{x}+\mu(\infty,x)n(x)$ is $k$-extreme. Furthermore it follows from the first part of the proof that for the constants $C_1,C_2,\dots, C_{k+1}$, not all equal to zero and such that
\[
\sum_{i=1}^{k+1} C_i\mu\left(u^*x_is(x)+\mu(\infty,x_i)n(x)\right)=\sum_{i=1}^{k+1} C_i\mu(x_i)=0,
\]
we have the corresponding equality for operators with the same constants $C_i$'s, that is
\[
\sum_{i=1}^{k+1} C_i\left(u^*x_is(x)+\mu(\infty,x_i)n(x)\right)=0.
\]
Since clearly $\sum_{i=1}^{k+1} C_i\mu(\infty,x_i)=0$, we have in fact that
\[
\sum_{i=1}^{k+1} C_iu^*x_is(x)=0.
\]
Recall that $s(x^*)=uu^*$, $u^*=s(x)u^*=u^*uu^*=u^*s(x^*)$. Multiplying the above sum by $u$ from the left and $u^*$ from the right we get
\[\sum_{i=1}^{k+1} C_is(x^*)x_iu^*s(x^*)=s(x^*)\left(\sum_{i=1}^{k+1} C_ix_iu^*\right)s(x^*)=0.
\]
 Thus $\sum_{i=1}^{k+1} C_ix_iu^*s(x^*)=\sum_{i=1}^{k+1} C_ix_iu^*=0$, and
\begin{equation}
\label{eq:noncom5.7}
\sum_{i=1}^{k+1} C_ix_is(x)=0.
\end{equation}
Observe that  $x^*=\frac1{k+1}\sum_{i=1}^{k+1}x_i^*$ and $\abs{x^*}=u^*x^*$. Repeating the same argument as above for $x^*$ and $x_i^*$'s instead of $x$ and $x_i$'s, respectively, and using the complex conjugate of the equality (\ref{eq:noncom5.6}),  $\sum_{i=1}^{k+1}\overline{C_i}\mu(x_i)=0$ we get $\sum_{i=1}^{k+1} \overline{C_i}x_i^*s(x^*)=0$. Hence
\begin{equation}
\label{eq:noncom5.8}
\sum_{i=1}^{k+1} C_is(x^*)x_i=0.
\end{equation}
Consequently combining (\ref{eq:noncom5.7}) and (\ref{eq:noncom5.8}), 
\[
\sum_{i=1}^{k+1} C_ix_i=\sum_{i=1}^{k+1} C_in(x^*)x_in(x).
\]
 Since  by  \cite[Lemma 3.3]{moj} the assumption $n(x)\M n(x^*)=0$ implies that $n(x)S\Mtau n(x^*)=0$, so $n(x^*)x_in(x)=0$ for all $i=1,2,\dots,k+1$. Therefore
\[
\sum_{i=1}^{k+1} C_ix_i=0,
\]
 and $x$ is a $k$-extreme point of $B_{\nonsp}$.
 
Let us suppose now that $\M$ is not non-atomic, $\mu(x)$ is $k$-extreme, and (i) and (ii) hold. Consider a non-atomic von Neumann algebra $\mathcal{A}$ with the trace $\kappa$, discussed in Remark \ref{rm2}. Then $\mathds{1}\overline\otimes x\in S(\mathcal{A},\kappa)$, $\tilde{\mu}(\mathds{1}\overline\otimes x)=\mu(x)$, and (i), (ii) are satisfied for the operator  $\mathds{1}\overline\otimes x$ (see the proof of Proposition \ref{lm:3}). Hence by the first part of the proof, $\mathds{1}\overline\otimes x$ is a $k$-extreme point of $B_{E(\mathcal{A},\kappa)}$. Since $\mathds{1}\overline\otimes x=\tilde{\pi}(x)$, where $\tilde{\pi}$ is an isometry from $\nonsp$ onto the subspace $E(\Complex \mathds{1}\otimes\M,\kappa)$ of $E(\mathcal{A},\kappa)$, it follows easily that $x$ is a $k$-extreme point of $B_{\nonsp}$.
\end{proof}

Combining now the results of Theorems \ref{thm:noncom4} and \ref{thm:noncom5}, we give a complete characterization of $k$-extreme points in terms of their singular value functions, when $\M$ is a non atomic von Neumann algebra. For $k=1$ we obtain the well-known theorem on extreme points proved in \cite{CKSextreme}.
\begin{theorem}
\label{thm:main}
Let $E$ be a strongly symmetric space on $[0,\tauone)$ and $\M$ be a non-atomic, semifinite von Neumann algebra with a faithful, normal, $\sigma$-finite trace $\tau$. An operator $x$ is a $k$-extreme point of $B_{\nonsp}$ if and only if $\mu(x)$ is a $k$-extreme point of $B_E$ and one of the following, not mutually exclusive, conditions holds:
\begin{itemize}
\item [{(i)}] $\mu(\infty,x)=0$;
\item[{(ii)}] $n(x)\mathcal{M}n(x^*)=0$ and $\abs{x}\geq\mu(\infty,x)s(x)$.
\end{itemize}
\end{theorem}

 Since in the commutative settings for any operator $x$, $n(x)=n(x^*)$, the conditions $n(x)\M n(x^*)=0$ and $\abs{x}\geq\mu(\infty,x)s(x)$  reduce to $\abs{x}\geq \mu(\infty,x)\one$. Therefore by the above theorem applied to the commutative von Neumann algebra $\M=L_{\infty}[0,\tauone)$ the following holds.

\begin{corollary}
\label{cor:2}
Let $E$ be a strongly symmetric function space. The following conditions are equivalent:
\begin{itemize}
\item [{(i)}]  $f$ is a $k$-extreme point of $B_E$;
\item[{(ii)}] $\mu(f)$ is a $k$-extreme point of $B_E$ and $\abs{f}\geq \mu(\infty,f)$.
\end{itemize}
\end{corollary}

The following simple observation will be useful in relating $k$-rotundity of $E$ and $\nonsp$.

\begin{lemma}
\label{lm:ezero}
If $E$ is a $k$-rotund symmetric function space then $E=E_0$, that is $\mu(\infty,f)=0$ for all $f\in E$. 
\end{lemma}
\begin{proof}
Suppose to the contrary that $E\neq E_0$, and so $\chi_{(0,\infty)}\in E$. Without loss of generality we can assume that $\norme{\chi_{(0,\infty)}}=1$. 
Let $f=\chi_{(k+1,\infty)}$. Then $\mu(f)=\chi_{(0,\infty)}$ and $f\in S_E$. For $i=1,2,\dots, k$, define 
\[
u_i=-\frac1k\chi_{(0,1)}+\frac1k\chi_{(i,i+1)}.
\]
Clearly, $\mu(f+u_i)=\chi_{(0,\infty)}$, and so $f+u_i\in S_E$. Moreover, 
\begin{align*}
\abs{f-\sum_{i=1}^k u_i}&=\abs{\chi_{(0,1)}-\frac1k\chi_{(1,k+1)}+\chi_{(k+1,\infty)}}=\chi_{(0,1)}+\frac1k\chi_{(1,k+1)}\\
&+\chi_{(k+1,\infty)}\leq \chi_{(0,\infty)},
\end{align*}
and also $f-\sum_{i=1}^k u_i\in B_E$. However $u_1,u_2,\dots, u_k$ are linearly independent, which in view of Proposition \ref{lm:2} implies that $f$ cannot be $k$-extreme.
\end{proof}
\begin{corollary}
\label{cor:global}
Let $\M$ be a semi-finite von Neumann algebra, with a faithful, normal, semi-finite trace $\tau$. If a symmetric space $E$ is $k$-rotund then $E(\mathcal{M},\tau)$ is $k$-rotund. If in addition $\M$ is non-atomic, then $k$-rotundity of $\nonsp$ implies $k$-rotundity of $E$.
\end{corollary}
\begin{proof}
If $E$ is $k$-rotund, then by Lemma \ref{lm:ezero} we have that  $E=E_0$. Let $x\in S_{\nonsp}$ and $x=\frac1{k+1}\sum_{i=1}^{k+1}x_i$, $x_i\in B_{\nonsp}$. Then $\mu(x)$ is $k$-extreme. Since $E=E_0$, $s(x)$ and $s(x_i)$, $i=1,2,\dots, k+1$ are $\sigma$-finite projections. Set $p=\vee_{i=1}^{k+1}s(x_i)\vee s(x)\vee s(x_i^*)\vee s(x^*)$. Then $pxp=ps(x^*)xs(x)p=s(x^*)xs(x)=x$ and $px_ip=x_i$, $i=1,2,\dots,k+1$.  Hence $x,x_i$, $i=1,2,\dots, k+1$, belong to the subspace which is isometric to $E(\M_p,\tau_p)$, where $\tau_p$ is $\sigma$-finite. By Theorem \ref{thm:noncom5} (i), we get that $x$ is $k$-extreme in $E(\M_p,\tau_p)$ and $x_1,x_2,\dots, x_{k+1}$ are linearly dependent in $E(\M_p,\tau_p)$. Hence $x_1,x_2,\dots, x_{k+1}$ are linearly dependent in $E(\M,\tau)$, and so $x$ is $k$- extreme point of $B_{\nonsp}$. Consequently $\nonsp$ is $k$-rotund.

 Suppose now that $\M$ is non-atomic and $E\Mtau$ is $k$-rotund. Then $E(\M_p,\tau_p)$ is $k$-rotund for any projection $p\in P(\M)$. Let $p\in P(\M)$ be a $\sigma$-finite projection with $\tau(p)=\tauone$ (see e.g. \cite[Lemma 1.13]{doctth}). By Proposition \ref{isom1}, $E$ is isometrically embedded in $E(\M_{p},\tau_{p})$ and therefore $E$ is also $k$-rotund.
\end{proof}

\section{Orbits and Marcinkiewicz spaces}
We finish with a characterization of $k$-extreme points in the orbits of functions.
Letting $g\in L_1[0,\alpha)+L_{\infty}[0,\alpha)$, $\alpha\leq \infty$,  the \textit{orbit} of $g$ is the set $\Omega(g)=\{f\in L_1[0,\alpha)+L_{\infty}[0,\alpha):\, f\prec g\}$ \cite{Ryff}. Clearly the inequality $f\prec g$ is equivalent to 
\[
\|f\|_{M_G}:=\sup_{t>0}\frac{\int_0^t \mu(f)}{\int_0^t \mu(g)}\leq 1.
\]
Setting $G(t)=\int_0^t \mu(g)$, the \textit{Marcinkiewicz} space $M_G$ is the set of all $f\in L^0$ such that $\|f\|_{M_G}<\infty$ \cite{KP, KPS}. The space $M_G$ equipped with the norm $\|\cdot \|_{M_G}$ is a strongly symmetric function space. Therefore the orbit $\Omega(g)$ is the unit ball $B_{M_G}$ in the space $M_G$.
\begin{theorem}
\label{thm:orbits}
Let $g\in L_1[0,\alpha)+L_{\infty}[0,\alpha)$. Then the following are equivalent.
\begin{itemize}
\item[{(i)}] $f$ is an extreme point of $\Omega(g)$.
\item[{(ii)}] $f$ is a $k$-extreme point of $\Omega(g)$.
\item[{(iii)}] $\mu(f)$ is a $k$-extreme point of $\Omega(g)$ and $\abs{f}\geq \mu(\infty,f)$.
\item[{(iv)}] $\mu(f)=\mu(g)$ and $\abs{f}\geq \mu(\infty,f)$.
\end{itemize}
\begin{proof}
Clearly (i) implies (ii), and  (ii) and (iii) are equivalent by Corollary \ref{cor:2}. The implication (iii) to (iv) follows by Corollary \ref{cor:orbitmain}. 
We will show next that $\mu(g)$ is an extreme point of $\Omega(g)$. Consequently if (iv) holds, $\mu(f)=\mu(g)$ is an extreme point of $\Omega(f)$ and by Corollary \ref{cor:2}, (i) follows.

Let $\mu(g)=\frac12 h_1+\frac12 h_2$, where $h_1,h_2\in \Omega(g)$. Then for all $s\in(0,\alpha)$ we have that 
\begin{align*}
\int_0^s\mu(g)&=\frac12\int_0^s h_1 +\frac12\int_0^s h_2\leq \frac12\int_0^s \mu(h_1)+\frac12\int_0^s \mu(h_2)\\
&\leq \frac12\int_0^s\mu(g)+
\frac12\int_0^s\mu(g)=\int_0^s \mu(g).
\end{align*}

Hence $h_1=\mu(h_1)=\mu(g)=\mu(h_2)=h_2$, and $\mu(g)$ is an extreme point of $\Omega(g)$.
\end{proof}
\end{theorem}

As an immediate consequence we get the following result.
\begin{corollary}
Let $M_G$ be the Marcinkiewicz space and $k$ be any natural number. The function  $f$ is a $k$-extreme point of $B_{M_G}$ if and only if  $\mu(f)=\mu(g)$ and $|f|\geq \mu(\infty,f)$. Consequently $f$ is a $k$-extreme point of $B_{M_G}$ if and only if $f$ is an extreme point of $B_{M_G}$.
\end{corollary}

We conclude the paper with a description of $k$-extreme points for another important class of orbits  $\Omega'(g)$, $0\leq g\in L_1[0,\alpha)$, $\alpha<\infty$ \cite{R, SZ}.

Recall that, given $0\leq g\in L_1[0,\alpha)$, $\alpha<\infty$, the orbit $\Omega'(g)$ is defined as 
\[
\Omega'(g)=\{0\leq f\in L_1[0,\alpha):\,f\prec g\quad \text{and}\quad\|f\|_1=\|g\|_1 \}.
\]

\begin{lemma}
\label{lm:primeorbit}
Let $0\leq g\in L_1[0,\alpha)$, $\alpha<\infty$. Then $f$ is a $k$-extreme point of $\Omega'(g)$ if and only if $f\geq 0$ and $f$ is a $k$-extreme point of $\Omega(g)$.
\end{lemma}
\begin{proof}
Suppose that $f\geq 0$ is a $k$-extreme point of  $\Omega(g)$. Then by Theorem \ref{thm:orbits}, $\mu(f)=\mu(g)$ and so $f\in\Omega'(g)$. Consequently, $f$ is a $k$-extreme point of $\Omega'(g)$.

Assume now that $f$ is a $k$-extreme point of $\Omega'(g)$. Let $f=\frac1{k+1}\sum_{i=1}^{k+1}f_i$ where $f_i\in \Omega (g)$. Since $f\geq 0$, $\int_0^\alpha f=\int_0^\alpha g$ and $f_i\prec g$, we have
\begin{align*}
\int_0^\alpha g=\int_0^\alpha f=\int_0^\alpha \mu(f)\leq \frac1{k+1}\sum_{i=1}^{k+1} \int_0^\alpha \mu(f_i)\leq \int_0^\alpha \mu(g)=\int_0^\alpha g.
\end{align*}
Hence, $\int_0^\alpha\mu(f_i)=\int_0^\alpha g$, for $i=1,2,\dots, k+1$.

Since $0\leq f=\frac1{k+1}\sum_{i=1}^{k+1}f_i
\leq\frac1{k+1}\sum_{i=1}^{k+1}\abs{f_i}$ and
 \begin{align*}
 \int_0^\alpha g&=\int_0^\alpha f=\int_0^\alpha\left(\frac1{k+1}
\sum_{i=1}^{k+1}f_i\right)\leq
\int_0^\alpha\left(\frac1{k+1}\sum_{i=1}^{k+1}
\abs{f_i}\right) \\
&=\frac1{k+1}\sum_{i=1}^{k+1}\|f_i\|_1=\int_0^\alpha g,
 \end{align*}
it follows that $f_i=\abs{f_i}$ and $f_i\in \Omega'(g)$, $i=1,2,\dots,k+1$. Therefore $f_1,f_2,\dots, f_{k+1}$ are linearly dependent and $f$ is a $k$-extreme point of $\Omega(g)$.  
\end{proof}

The above lemma and Theorem \ref{thm:orbits} show that the sets of extreme and $k$-extreme points for $\Omega'(g)$, $0\leq g\in L_1[0,\alpha)$, $\alpha<\infty$, coincide. Consequently, the description of extreme points of $\Omega'(g)$ presented in \cite{SZ} applies also for $k$-extreme points. 
\begin{proposition}
\label{thm:orbits1}
 Let $0\leq g\in L_1[0,\alpha)$, $\alpha<\infty$. Then $f$ is a $k$-extreme point of $\Omega'(g)$ if and only if  $\mu(f)$ is a $k$-extreme point of $\Omega'(g)$. Moreover, the set of all $k$-extreme points of $\Omega'(g)$ is given by 
 \[
 \text{$k$-ext}\left(\Omega'(g)\right)=\{0\leq f\in L_1[0,\alpha):\, \mu(f)=\mu(g)\}. 
 \]

\end{proposition}
\bibliographystyle{amsplain}

\end{document}